\documentclass[10pt]{article} %{amsart}
\usepackage{amsmath,amsfonts,amssymb}
\usepackage{mathtools}
\usepackage[english]{babel}
\usepackage[autolang=other,backend=bibtex,bibencoding=ascii,maxcitenames=4,mincitenames=1,style=trad-alpha]{biblatex}
\usepackage{color}
\usepackage{csquotes}
\usepackage{enumerate}
\usepackage{esint}
\usepackage[T1]{fontenc}
\usepackage[pdftex]{hyperref}
\usepackage{mathrsfs}
\usepackage{graphicx}
\usepackage[dvipsnames]{xcolor}
\usepackage{pgfplots}
\usepackage{subfigure}

\usepackage[charter]{mathdesign}

\renewcommand{\vec}[1]{\boldsymbol{#1}}
\renewcommand{\div}{\operatorname{div}}
\newcommand{\grad}{\vec{\nabla}}

\newcommand{\lifting}{g}

\providecommand{\abs}[1]{\left\lvert#1\right\rvert}
\providecommand{\norm}[1]{\lVert#1\rVert}
\providecommand{\dnorm}[1]{\left\lVert#1\right\rVert}

\newtheorem{theorem}{Theorem}[section]

\newtheorem{lemma}[theorem]{Lemma}
\newtheorem{proposition}[theorem]{Proposition}

\newtheorem{remark}{Remark}

\newenvironment{proof}{\textsc{\noindent Proof. \ignorespaces}}{\nopagebreak\hspace*{\fill}$\square$}

\hypersetup{
   colorlinks,
   linkcolor=black,
   citecolor=black,
   urlcolor=black
}

\AtEveryBibitem{\clearlist{language}} % clears language
\title{On the method of reflections}
\author{Philippe Laurent\thanks{Institut de Recherche en Communications et Cybern\'etique de Nantes, CNRS, UMR 6597, \'Ecole des Mines de Nantes, 4, rue Alfred Kastler, 44307 Nantes, France (\url{philippe.laurent@mines-nantes.fr}).}, Guillaume Legendre\thanks{CEREMADE, UMR CNRS 7534, Universit\'e Paris-Dauphine, Universit\'e PSL, Place du Mar\'echal De Lattre De Tassigny, 75775 Paris cedex 16, France (\url{guillaume.legendre@dauphine.fr}).}, Julien Salomon\thanks{INRIA Paris, ANGE project-team, 2 rue Simone Iff, 75589 Paris cedex 12, France. Sorbonne Universit\'es, Laboratoire Jacques-Louis Lions (\url{julien.salomon@inria.fr}).}}
\bibliography{reflection}

\begin{document}
\maketitle
\begin{abstract}
This paper aims at reviewing and analysing the method of reflections, which is an iterative procedure designed for solving linear boundary value problems set in multiply connected domains. Being based on a decomposition of the domain boundary, this method is particularly well-suited to numerical solvers relying on boundary integral representation. For both the sequential and parallel forms of the method appearing in the literature, we interpret the procedure in terms of projection operators. Using a Hilbert space setting and orthogonality, we prove the unconditional convergence of the sequential form and propose a modification of the parallel one that makes it unconditionally converging. Several examples of boundary value problems that enter such a framework are given, an alternative proof of convergence is provided in a case which does not. A few numerical tests conclude the study.
\end{abstract}

\tableofcontents

\section{Introduction}
In 1911, inspired by the method of image charges in electromagnetism, Smoluchowski \cite{Smoluchowski:1911} introduced an iterative process to compute the hydrodynamic forces exerted on an assemblage of an arbitrary number of spheres falling in an unbounded viscous fluid. Later dubbed the \emph{method of reflections}, this technique has subsequently been featured prominently in reference textbooks on hydrodynamics at low Reynolds number, like those by Happel and Brenner \cite{Happel:1983}, Kim and Karrila \cite{Kim:1991} or Dhont \cite{Dhont:1996}, and employed in several articles dealing with the motion of particles immersed in a viscous fluid (see for instance \cite{Kynch:1959,Jones:1978,Chen:1988,Ichiki:2001,Wilson:2013}). It may be described as a systematic scheme by which a (generally exterior) linear boundary value problem associated with several ``objects'' (e.g., particles in the case of a suspension) may be solved by computing and summing solutions of boundary value problems involving only a single object, which are called \emph{reflections}. Physically speaking, one may visualise its very principle by \emph{``supposing an initial disturbance to be reflected from the boundaries involved and to produce succeedingly smaller effect with each successive reflection''} \cite[chapter 8]{Happel:1983}, hence its given name.

In a more formal manner, the method exploits the superposition principle, inherent to any linear system, to follow a divide-and-conquer strategy, assuming that each of the single-object problems is somewhat easier to solve than the many-object one. In this sense, it bears some similarities to Schwarz-type domain decomposition methods (see \cite{Gander:2008} for instance), which attempt to solve a (generally interior) boundary value problem by splitting it into several boundary value problems set on different, at most partially overlapping, subdomains and iterating to coordinate the solution between adjacent subdomains. There is, however, an important conceptual difference between these two approaches. Indeed, any single object problem to be solved in the method of reflections is defined on a (possibly unbounded) domain that is the interior of the complement of the considered object and thus contains the whole domain in which the main problem is set. This technique must therefore rather be seen as a \emph{boundary} decomposition method.

The method of reflections was historically devised with human computers in mind, as well as tractable and explicit formulas for the reflections, which implied that the objects were few and identical, with a shape such that available analytical forms of the Stokes and Fax\'en laws allowed the practical approximation of the velocity field by a truncated series. With time, these restrictions, coupled with the fact that a large number of terms may be required in the truncated expansion to obtain an accurate approximation (most notably when the objects are close to each other), made its use less attractive. Nevertheless, since the method in itself does not rely on the manner in which the boundary value problems are solved, it became clear, with the advent of electronic computers, that methods using a numerical discretization (more advantageously those based on boundary integral representation) could be employed, and even combined\footnote{One can for instance check the reference \cite{Boyer:2012}, in which different solution methods are used depending on the nature of the considered one-object problem.} with the aforementioned analytical methods, to efficiently address problems involving numerous objects of arbitrary shape, size or type of imposed boundary condition.

Despite a number of papers concerned with applications (see \cite{Ichiki:2001,Mettot:2011,Boyer:2012,Wilson:2013} for instance), it appears that the method of reflections has seldom been studied from a mathematical perspective. One must cite the pioneering work of Luke \cite{Luke:1989}, in which the convergence of the original form of the method applied to the solution of the so-called mobility problem in a bounded domain is investigated theoretically. By formulating the problem in an appropriate Hilbert space setting and interpreting the method in terms of orthogonal projection operators, Luke proved that the method always converges to the solution of the problem and that its convergence rate is linear if a geometric property between the subspaces associated with the projection operators holds. More recently, Traytak \cite{Traytak:2006} studied a variant of the method applied to the solution of a Dirichlet problem in a three-dimensional unbounded domain complementary to a set of spheres. Resorting to analytical techniques, he established necessary and sufficient conditions of convergence with respect to the radii of the spheres and the distances between their respective centres. These two important contributions notwithstanding, it appears that a fully developed mathematical theory of the method of reflections is still lacking.

Furthermore, we may add that a few missed facts have led to some misunderstandings on the part of the scientific communities in which the method is used. For instance, when Ichiki and Brady \cite{Ichiki:2001} state, citing the work of Luke, that \textit{``while the convergence of the method of reflections has been proven for the mobility problem, the convergence for the resistance problem is an open question.''} and subsequently give a ``counter-example'' to the convergence of the method for configurations involving more than two objects, they somehow undermine their discovery by failing to notice that the method of reflections considered in their work differs from the one studied by Luke. %Also, while the convergence properties of the method depend on the type of boundary value problem under consideration, conclusions have sometimes been wrongly derived based on the sole nature of the problem. For instance, Luke \cite{Luke:1989} %(see page 1650)
% admits that his convergence proof in the case of the mobility problem fails for the resistance problem, and, in \cite{Wilson:2013}, the author writes from a numerical perspective: \textit{``Clearly in a philosophical sense these two problems are equivalent: but numerically we have found empirically that their convergence behaves very differently. For well-separated spheres there are no convergence issues in either problem, but when the spheres are close together the resistance problem is found to converge much worse than the mobility problem.''} ***
%  However, as shown in the present work, the convergence of the method for both problems can be proved with the same mathematical tools. Likewise, the fact that the boundary value problem to be solved is set in a bounded domain or in an unbounded one, has also been pointed out as modifying the convergence theory of the method. In \cite{Luke:1989}% (see page 1650)
%, one can read that: \textit{``One of the most glaring restrictions in the hypothesis of the convergence proof is that the container [...] is bounded. Physical intuition suggests no reason why the presence of container walls would aid the reflection method. But a rigorous proof is lacking.''} Again, the convergence of the method can be shown for a boundary value problem set on an unbounded domain as well as one set on a bounded domain.
It was also empirically observed by Wilson \cite{Wilson:2013} that the convergence behaviour of the method may depend significantly on the type of boundary value problem considered, and Luke stated that his convergence proof does not extend to other types of boundary conditions or in the case of an unbounded domain, leaving open a number of questions. In the present work, we show that convergence can indeed be established for several kinds of boundary value problems using the same mathematical tools. By doing so, we aim at broadening the original analysis of Luke and propose a unified theoretical framework for the analysis of the existing forms of the method of reflections. %Secondly, we obtain results of convergence and explain the cases of divergence observed in some applications. By doing so, we hope to clear the confusion which appears to surround the method and, perhaps, introduce it to a wider audience.

\smallskip

The outline of the paper is the following. Using a well-known analogy between Stokesian hydrodynamics and electrostatics, we first review in Section \ref{section: presentation of the method of reflections} the different instances of the method of reflections found in the literature on a couple of toy problems involving the Laplace operator. This exposition leads to a more formal presentation of two versions of the method, applied to the solution of an abstract boundary value problem. The first form, called \emph{sequential} in the present work, appears to be the original one from an historical perspective. The second one, which we call \emph{parallel}, is, to the best of our knowledge, a variant independently introduced by Golusin in 1934 \cite{Golusin:1934} and seems to be the form of the method featured the most prominently in the literature. In Section \ref{framework}, the method is recast in a loose functional framework, which allows to derive some algebraic properties of the iterative process and to interpret it in terms of projection operators. Section \ref{section: convergence analysis} is devoted to the convergence analysis of the method in a Hilbert space setting. Unconditional convergence results  are obtained, for both the sequential form and a modification of the parallel form of the method, whenever the considered boundary value problem can be solved by the \emph{method of orthogonal projection} \cite{Weyl:1940,Vishik:1949}, and several examples of applications entering such a context are given. Finally, some numerical experiments are presented in Section \ref{Numerics}. 

\section{Two examples of applications from the literature}\label{section: presentation of the method of reflections}
In this section, we describe two forms of the method of reflections commonly found in the literature. The method is indeed typically considered in the context of hydrodynamics at low Reynolds number to solve two distinct types of problems for flows involving hydrodynamic interactions among particles, respectively called the \emph{mobility problem} and the \emph{resistance problem} (see \cite{Kim:1991} for instance). Both are boundary value problems based on the Stokes equations for an incompressible flow of a Newtonian fluid. In the former, the forces and torques are to be determined for specified particle velocities in the ambient fluid, while in the latter the particles' forces and torques are prescribed in the ambient fluid and the velocities are unknown. However, to simplify the presentation, we exploit an existing analogy between Stokesian hydrodynamics and electrostatics, recalled by Luke \cite{Luke:1989}, in which the Stokes equations are replaced by the Laplace equation, and the fluid velocity by the electrostatic potential.

\subsection{An electrostatic analogue of the mobility problem: the sequential form}\label{subsec:introMob}
The method of reflections, as introduced by Smoluchowski \cite{Smoluchowski:1911}, was studied by Luke \cite{Luke:1989} when used to solve the mobility problem in hydrodynamics. The electrostatic analogue of this problem consists in determining the electric potential in a container $\Omega$ in $\mathbb{R}^d$ ($d$ being an integer strictly greater than $2$), such that its value is constant, but unknown, on the surfaces of $N$ conducting objects $O_j\subset\Omega$, $j$ in $\{1,\dots,N\}$, and the corresponding amount of electric charge on each of these surfaces is known, that is: \emph{given $N$ real numbers $Q_j$, $j$ in $\{1,\dots,N\}$, find the scalar function $u$ and the $N$ real numbers $c_j$, $j$ in $\{1,\dots,N\}$, which verify}
\begin{align}
&-\Delta u=0\text{ in }\Omega\setminus\cup_{j=1}^N\overline{O_j},\label{Laplace equation, 4th type problem}\\
&u=c_j\text{ on }\partial O_j,\ j=1,\dots,N,\label{boundary conditions, 4th type problem}\\
&\displaystyle\int_{\partial O_j}\frac{\partial u}{\partial\vec{n}}(x)\,\mathrm{d}S(x)=Q_j,\ j=1,\dots,N,\label{surface charge conditions, 4th type problem}
\end{align}
where $\Delta$ denote the Laplace operator, $\frac{\partial u}{\partial\vec{n}}$ is the trace of the normal derivative of $u$ on the considered boundary, $\vec{n}$ being the outward-pointing unit normal vector to the boundary, and $\mathrm{d}S$ stands for the volume form on each of the considered hypersurfaces.

Such a non-local boundary condition is sometimes said to be of the \emph{fourth type\footnote{One may consider the first kind of boundary condition for the Laplace equation to be Dirichlet's, the second, Neumann's, and the third, Robin's.}} \cite[Chapter III]{Showalter:1977} or called an \emph{equivalued surface boundary condition} or a \emph{total flux boundary condition} \cite{Li:1989}. It arises in various applications (see the examples given in~\cite{Li:1989}).

Note that the domain $\Omega$ may or may not be bounded. In the latter case, one should prescribe an additional condition at infinity for the problem to be well-posed. In what follows, we assume that the container is bounded and that a homogeneous Dirichlet condition is imposed on its boundary:
\begin{equation}\label{outer Dirichlet condition, 4th type problem}
u=0\text{ on }\partial\Omega.
\end{equation}

In order to solve the resulting boundary value problem, the method of reflections generates a sequence $(u^{(k)})_{k\in\mathbb{N}}$ of approximations to the solution, starting from an initial field $u^{(0)}$ satisfying equations \eqref{Laplace equation, 4th type problem}, \eqref{surface charge conditions, 4th type problem} and \eqref{outer Dirichlet condition, 4th type problem} (but not necessarily equation \eqref{boundary conditions, 4th type problem}), and next cyclically correcting the boundary value of the approximation on each of the boundaries of the set of objects. Such a chore is achieved through the introduction of auxiliary fields\footnote{The algorithm presented here is a reformulation of the method analysed in \cite{Luke:1989}, as the auxiliary fields do not appear explicitly in the reference.} (the so-called \emph{reflections}), defined recursively as the solutions of single-object problems. For the problem at hand, these take the form of the sequences of functions $(u^{(k)}_i)_{k\in\mathbb{N}^*}$ and scalars $(c^{(k)}_i)_{k\in\mathbb{N}^*}$, with $i$ in $\{1,\dots,N\}$, such that
\[
\forall i\in\{1,\dots,N\},\ \left\{\begin{array}{l}
-\Delta u^{(1)}_i=0\text{ in }\Omega\setminus\overline{O_i},\\
u^{(1)}_i=c^{(1)}_i-u^{(0)}-\sum_{j=1}^{i-1}u^{(1)}_j\text{ on }\partial O_i,\\
\displaystyle\int_{\partial O_i}\frac{\partial u^{(1)}_i}{\partial\vec{n}}(x)\,\mathrm{d}S(x)=0,\\
u^{(1)}_i=0\text{ on }\partial\Omega,
\end{array}\right.
\]
and
\[
\forall k\in\mathbb{N}^*,\ \forall i\in\{1,\dots,N\},\ \left\{\begin{array}{l}
-\Delta u^{(k+1)}_i=0\text{ in }\Omega\setminus\overline{O_i},\\
u^{(k+1)}_i=c^{(k+1)}_i-\sum_{j=1}^{i-1}u^{(k+1)}_j-\sum_{j=i+1}^{N}u^{(k)}_j\text{ on }\partial O_i,\\
\displaystyle\int_{\partial O_i}\frac{\partial u^{(k+1)}_i}{\partial\vec{n}}(x)\,\mathrm{d}S(x)=0,\\
u^{(k+1)}_i=0\text{ on }\partial\Omega.
\end{array}\right.
\]
The approximation of the solution after the $\ell$th cycle of the method is then defined by summing the restrictions of the reflections to $\Omega\setminus\cup_{j=1}^N\overline{O_j}$ as
follows,
\begin{equation}\label{lth cycle approximation}
u^{(\ell)}=u^{(0)}+\sum_{k=1}^\ell\sum_{i=1}^Nu^{(k)}_i\text{ in }\Omega\setminus\cup_{j=1}^N\overline{O_j}.
\end{equation}

We observe that the imposed surface charges are taken into account in the initial approximation $u^{(0)}$ and never subsequently modified by the reflections, which only aim at correcting the boundary values one at a time. The practical construction of such a first approximation does not violate the single-object paradigm of the method, since this field may be obtained as the sum of $N$ solutions to problems of that type by considering each conducting object separately. We also note that the boundary datum for each of the single-object problems defined above depends at a given cycle on quantities computed during both the current cycle and the previous one (or on the initial approximation during the first cycle), which necessarily implies \emph{successive} computations of their respective solutions, in the spirit of the Gauss--Seidel method to solve linear systems of equations. Hence, we call hereafter this procedure the \emph{sequential form} of the method of reflections.

Luke \cite{Luke:1989} showed the unconditional convergence of this version of the method of reflections applied to solve the mobility problem for the Stokes equations (see Subsection \ref{appl:Luke}). Similar boundary decomposition methods were proposed by Balabane and Tirel \cite{Balabane:1997} to solve the Helmholtz equation outside a union of obstacles without trapping rays, and by Coatl\'even and Joly \cite{Coatleven:2012} to effectively solve operator factorized forms of time-harmonic multiple-scattering problems in periodic media.

\subsection{An electrostatic analogue of the resistance problem: the parallel form}\label{subsec:introRes}
From a mathematical standpoint, the so-called resistance problem of hydrodynamics is a boundary value problem for the Stokes equations with non-homogeneous Dirichlet boundary conditions. Its electrostatic analogue thus amounts to determining the electric potential in a container $\Omega$, its values on the surfaces of $N$ conducting objects $O_j\subset\Omega$, $j$ in $\{1,\dots,N\}$, being known, that is: \emph{given $N$ scalar functions $U_j$, $j$ in $\{1,\dots,N\}$, respectively defined on $\partial O_j$, $j$ in $\{1,\dots,N\}$, find the scalar function $u$ satisfying}
\begin{align}
&-\Delta u=0\text{ in }\Omega\setminus\cup_{j=1}^N\overline{O_j},\label{Laplace equation, Dirichlet problem}\\
&u=U_j\text{ on }\partial O_j,\ j=1,\dots,N.\label{boundary conditions, Dirichlet problem}
\end{align}
As with the previous example, we assume that the domain\begin{flushright}

\end{flushright} $\Omega$ is bounded and impose a homogeneous Dirichlet condition on its boundary to complete the problem,
\begin{equation}\label{outer Dirichlet condition, Dirichlet problem}
u=0\text{ on }\partial\Omega.
\end{equation}

The second form of the method of reflections hinges on a different construction of the sequence of approximations $(u^{(k)})_{k\in\mathbb{N}}$. While the initial approximation $u^{(0)}$ has to satisfy equations \eqref{Laplace equation, Dirichlet problem} and \eqref{outer Dirichlet condition, Dirichlet problem} (but not necessarily equation \eqref{boundary conditions, Dirichlet problem}) and the subsequent approximations are still given by the sum \eqref{lth cycle approximation}, the auxiliary fields $(u_i^{(k)})_{k\in\mathbb{N}^*}$, with $i$ in $\{1,\dots,N\}$, are now defined as the respective solutions to the following single-object problems
\[
\forall i\in\{1,\dots,N\},\ \left\{\begin{array}{l}
-\Delta u^{(1)}_i=0\text{ in }\Omega\setminus\overline{O_i},\\
u^{(1)}_i=U_i-u^{(0)}\text{ on }\partial O_i,\\
u^{(1)}_i=0\text{ on }\partial\Omega,
\end{array}\right.
\]
and
\[
\forall k\in\mathbb{N}^*,\ \forall i\in\{1,\dots,N\},\ \left\{\begin{array}{l}
-\Delta u^{(k+1)}_i=0\text{ in }\Omega\setminus\overline{O_i},\\
u^{(k+1)}_i=-\sum_{j=1}^{i-1}u^{(k)}_j-\sum_{j=i+1}^Nu^{(k)}_j\text{ on }\partial O_i,\\
u^{(k+1)}_i=0\text{ on }\partial\Omega.
\end{array}\right.
\]
We observe here that the datum for each of the above problems depends at a given cycle on quantities computed during the previous cycle (and on the boundary data for the first cycle). The solutions of these problems may thus be carried out \emph{simultaneously}, like in the Jacobi method for the iterative solution of linear systems of equations, leading us to refer to this procedure as the \emph{parallel form} of the method of reflections.

This version of the method of reflections was devised by Golusin \cite{Golusin:1934} to constructively solve the Dirichlet problem for the Laplace equation in multiply connected circular domains. It appears to be the most commonly found in the literature (see for instance the presentation in subsection 214 of the textbook \cite{Smirnov:1964}). Concerning its convergence, Happel and Brenner \cite{Happel:1983} wrote\footnote{The version of the method presented in Chapter 6 of \cite{Happel:1983} differs slightly from the above procedure by focusing on one specific object. The reflections with respect to this object are then computed as in the sequential form of the method, whereas the reflections with respect to the other $N-1$ objects are computed as in the parallel form.} %, p. 236,
 that: \textit{``it must be pointed out that no rigorous proof exists that the iteration scheme converges to the desired solution.''}

As a matter of fact, a numerical example of divergence for a particular configuration is presented by Ichiki and Brady \textcolor{black}{in} \cite{Ichiki:2001}, the method being \textcolor{black}{there} used to solve the Stokes resistance problem in the presence of rigid spherical particles in an unbounded domain. The solutions of the single particle problems are numerically approximated by truncated multipole expansions, but the authors conjecture that the observed divergence is unrelated to the order of truncation of the expansions.

In \cite{Traytak:2006}, Traytak analysed the method applied to the solution of a Dirichlet problem for the Laplace equation in the unbounded complement of a set of spheres. \textcolor{black}{He} obtained necessary and sufficient conditions for its convergence and exhibited simple cases of divergence when the number of spheres is greater than or equal to eight.
 
More recently, conditional convergence results for the method were derived by H\"ofer and Vel\'azquez \cite{Hofer:2018} for an \emph{infinite} number of spherical particles, allowing for new proofs of classical homogenization results for the Dirichlet problem for both the Poisson and Stokes equations in perforated domains (see also \cite{Hoefer:2018,Niethammer} which both deal with the mobility problem in a similar context).

This version of the method was also used by Jabin and Otto \cite{Jabin:2004} to identify the dilute regime of a cloud of sedimenting particles, in which the particles do not significantly interact and sink as if they were isolated. Additionally, it may be seen that the boundary decomposition technique introduced\footnote{In the same work, a sufficient condition for convergence, depending on the frequency, the diameters and the areas of the sub-scatterers, as well as the respective distances between the sub-scatterers, is established.} by Balabane \cite{Balabane:2004} (see also \cite{Ganesh:2009,Wang:2013} for practical and numerical applications) for solving a boundary value problem involving the Helmholtz equation in the unbounded complement of a scatterer made of a union of disjoint sub-scatterers is identical to this parallel form of the method. Connections between the principle of this instance of the method and the approach known as the Foldy--Lax model\footnote{To see this, one can compare the hierarchy of different levels of approximation given in \cite{Cassier:2013}, ranging from the Born approximation to the Foldy--Lax model, with the sequence of approximations produced by the method of reflections.} \cite{Foldy:1945,Lax:1951,Lax:1952} or the generalised Born series technique~\cite{Schuster:1985} for the multiple scattering of waves may also be pointed out.

\section{General formulations}\label{framework}
In this section, we summarize and interpret the application of both forms of the method of reflections to the solution of an abstract linear boundary value problem. We want to emphasize that the principle of the method is very general and, because of this generality, the present treatment is necessarily formal. As a consequence, we deliberately refrain from specifying a particular functional framework in order to focus on the algebraic aspects of the method. A rigorous setting is provided in Section \ref{section: convergence analysis}, with explicit examples given in Subsection \ref{practical cases}.

\smallskip

We consider a simply connected regular open subset $\Omega$ of $\mathbb{R}^d$, and, given a positive integer $N$, a family of arbitrarily numbered ``objects'' or ``holes'', which are disjoint subsets $O_j\subset\Omega$, with $j$ in $\{1,\dots,N\}$, in the sense that they are bounded, simply connected, open sets with smooth boundaries such that their closures are non-overlapping. The open set $\Omega\setminus\cup_{j=1}^N\overline{O_j}$ is then called a \emph{perforated} domain. We are interested in solving the following non-homogeneous boundary value problem: \textit{find a field $u$ satisfying
\begin{eqnarray}
\label{BVP main equation}&&\mathcal{L}u=f\text{ in }\Omega\setminus\cup_{j=1}^N\overline{O_j}\\
\label{BVP main boundary conditions}&&\mathcal{B}_ju=b_j%\text{ on }\partial O_j
,\ j=1,\dots,N,
\end{eqnarray}
where $\mathcal{L}$ is a linear differential operator, acting on functions defined in $\Omega$, each operator $\mathcal{B}_j$, with $j$ in $\{1,\dots,N\}$, is a linear operator acting, possibly in a non-local manner, on functions defined on the boundary $\partial O_j$ of the $j$th object, and the functions $f$ and $b_j$, with  $j$ in $\{1,\dots,N\}$, respectively defined on $\Omega$ and $\partial O_j$, with  $j$ in $\{1,\dots,N\}$, are the data of the boundary value problem.} In the applications of the method previously recalled, the operator $\mathcal{L}$ is typically of elliptic type, but this is not necessarily the case (see \cite{Balabane:2004} for instance).

\begin{remark}\label{rem:farbitrary}
Note that the function $f$ is defined over the whole of $\Omega$, while the problem itself is defined over the perforated domain $\Omega\setminus\cup_{j=1}^N\overline{O_j}$, as customary in problems dealing with homogenisation. In the same way, the functions $b_j$, with $j$ in $\{1,\dots,N\}$, forming the rest of the data are usually related to a single function, globally defined over $\Omega$. % For instance, in the case of nonhomogeneous Dirichlet boundary conditions, the function $b_i$ will correspond to the trace on the boundary $\partial O_i$ of a single function set in $\Omega$.
We will make such an assumption in Subsection~\ref{projection setting}.
\end{remark}

If the domain $\Omega$ is bounded, a homogeneous condition on the boundary $\partial\Omega$ is added to the above system of equations. In the case where it is unbounded, this condition is replaced, or complemented, by one set at infinity. Either way, this boundary condition is denoted
\begin{equation}\label{BVP additional condition}
\mathcal{B}_0u=0,
\end{equation}
and it is assumed that the system \eqref{BVP main equation}-\eqref{BVP additional condition} defines a \emph{well-posed} (in the sense of Hadamard) boundary value problem in an appropriate function space, \textit{i.e.}, there exists a unique solution $u^*$ to the problem, which belongs to this space and depends continuously on the data. This implies in particular that the operators $\mathcal{B}_j$, with $j$ in $\{1,\dots,N\}$, verify certain admissibility conditions with respect to the operator $\mathcal{L}$ (see \cite{Lions:1972}).

\smallskip

To solve this boundary value problem using the sequential form of the method of reflections, one begins by computing an initial approximation $u^{(0)}$ satisfying
\begin{equation}\label{definition u0}
\left\{\begin{aligned}
&\mathcal{L}u^{(0)}=f\text{ in }\Omega,\\
&\mathcal{B}_0u^{(0)}=0.
\end{aligned}\right.
\end{equation}
An iterative (or cyclic) phase then follows, in which one solves recursively the single-object problems associated with the reflections given by
\begin{equation}\label{initseq}\
\forall i\in\{1,\dots,N\},\ \left\{\begin{aligned}
&\mathcal{L} u_i^{(1)}=0\text{ in }\Omega\setminus\overline{O_i},\\
&\mathcal{B}_iu_i^{(1)}=b_i-\mathcal{B}_i\left(u^{(0)}+\sum_{j=1}^{i-1}u_j^{(1)}\right)%\text{ on }\partial O_i
,\\
&\mathcal{B}_0u_i^{(1)}=0,
\end{aligned}\right. 
\end{equation}
and
\begin{equation}\label{iterseq}
\forall k\in\mathbb{N}^*,\ \forall i\in\{1,\dots,N\},\ \left\{\begin{aligned}
&\mathcal{L} u_i^{(k+1)}=0\text{ in }\Omega\setminus\overline{O_i},\\
&\mathcal{B}_iu_i^{(k+1)}=-\mathcal{B}_i\left(\sum_{j=1}^{i-1}u_j^{(k+1)}+\sum_{j=i+1}^Nu_j^{(k)}\right)%\text{ on }\partial O_i
,\\
&\mathcal{B}_0u_i^{(k+1)}=0,
\end{aligned}\right.
\end{equation}
in order to update the approximation of the solution using the formula
\begin{equation}\label{update}
\forall k\in\mathbb{N},\ u^{(k+1)}=u^{(k)}+\sum_{i=1}^N{u_i^{(k+1)}}\text{ in }\Omega\setminus\cup_{j=1}^N\overline{O_j}.
\end{equation}
The approximate solution after $k$ cycles is thus defined by a finite double sum
\begin{equation}\label{approximate solution}
\forall k\in\mathbb{N},\ u^{(k)}=u^{(0)}+\sum_{\ell=1}^k\sum_{i=1}^N{u_i^{(\ell)}}\text{ in }\Omega\setminus\cup_{j=1}^N\overline{O_j}.
\end{equation}

\medskip

We next turn to the parallel version of the method of reflections. In this case, while the initialisation is required to satisfy \eqref{definition u0} and the update of the approximate solution is given by~\eqref{update}, the collections of problems to be solved during the cycling phase are
\begin{equation}\label{initpar}
\forall i\in\{1,\dots,N\},\ \left\{\begin{aligned}
&\mathcal{L} u_i^{(1)}=0\text{ in }\Omega\setminus\overline{O_i},\\
&\mathcal{B}_iu_i^{(1)}=b_i-\mathcal{B}_iu^{(0)}%\text{ on }\partial O_i
,\\
&\mathcal{B}_0u_i^{(1)}=0.
\end{aligned}\right. 
\end{equation}
and
\begin{equation}\label{iterpar}
\forall k\in\mathbb{N}^*,\ \forall i\in\{1,\dots,N\},\ \left\{\begin{aligned}
&\mathcal{L} u_i^{(k+1)}=0\text{ in }\Omega\setminus\overline{O_i},\\
&\mathcal{B}_iu_i^{(k+1)}=-\mathcal{B}_i\left(\sum_{j=1}^{i-1}u_j^{(k)}+\sum_{j=i+1}^Nu_j^{(k)}\right)%\text{ on }\partial O_i
,\\
&\mathcal{B}_0u_i^{(k+1)}=0.
\end{aligned}\right. 
\end{equation}

\medskip

We observe that both versions of the method of reflections are well-defined as soon as the no-object problem \eqref{definition u0} and the respective single-object problems \eqref{initseq} and \eqref{iterseq}, or \eqref{initpar} and \eqref{iterpar}, are well-posed. Moreover, the method is said to converge if the sequence of functions defined by \eqref{approximate solution} has a limit as the number of achieved cycles $k$ tends to infinity.

\begin{remark}
In some cases, the well-posedness of the single-object problems to be solved during the cyclic phase is a direct consequence of the well-posedness of the many-object boundary value problem one wants to solve with the method. However, considering for instance a non-homogeneous pure Neumann problem for the Laplace operator, one easily sees that the compatibility conditions satisfied by the data of the problem could generally prevent similar conditions to be satisfied by the data of the sub-problems stemming from the boundary decomposition, leading to ill-posed boundary value problems when applying the method. As a consequence, the method of reflections is not applicable to \emph{any} linear boundary value problem, since the decomposition paradigm on which it is based may not be valid.
\end{remark}

\subsection{Algebraic properties of the reflections}\label{subsec:Alg}
The sequences of reflections $(u_i^{(k+1)})_{k\in\mathbb{N}}$, with $i$ in $\{1,\dots,N\}$, constructed by both forms of the method may be viewed as sequences of partial correctors to the sequence of approximate solutions $(u^{(k)})_{k\in\mathbb{N}}$ in the following sense. 

\begin{lemma}\label{correction formula lemma}
Suppose boundary value problem \eqref{BVP main equation}-\eqref{BVP additional condition} admits a lifting of its boundary data, \textit{i.e.}, there exists a function $\lifting$ defined on $\Omega$ such that $\mathcal{B}_jg=b_j$, $j=0,\dots,N$, and that the $N$ sequences $(u^{(k+1)}_j)_{k\in\mathbb{N}}$, $j=1,\dots,N$, produced either by the sequential or by the parallel form of the method of reflections are well defined. Then, one has
\begin{equation}\label{correction formula sequential}
\forall k\in\mathbb{N},\ \forall i\in\{1,\dots,N\},\ \mathcal{B}_i\left(u^{(k)}+\sum_{j=1}^iu_j^{(k+1)}\right)=\mathcal{B}_i\lifting%\text{ on }\partial O_i
,
\end{equation}
for the sequential form, and
\begin{equation}\label{correction formula parallel}
\forall k\in\mathbb{N},\ \forall i\in\{1,\dots,N\},\ \mathcal{B}_i\left(u^{(k)}+u_i^{(k+1)}\right)=\mathcal{B}_ig%\text{ on }\partial O_i
,
\end{equation}
for the parallel form, where the sequence $(u^{(k)})_{k\in\mathbb{N}}$ is given by \eqref{approximate solution}.
\end{lemma}
\begin{proof}
Both equalities are easily established by induction. First, for the sequential form, it stems from the linearity of the operators, problem \eqref{initseq}, and the definition of $\lifting$ that, for the base case $k=0$ and any integer $i$ in $\left\{1,\dots,N\right\}$,
\[
\mathcal{B}_i\left(u^{(0)}+\sum_{j=1}^iu_j^{(1)}\right)=\mathcal{B}_i\left(u^{(0)}+\sum_{j=1}^{i-1}u_j^{(1)}+u_i^{(1)}\right)=\mathcal{B}_i\left(u^{(0)}+\sum_{j=1}^{i-1}u_j^{(1)}-u^{(0)}-\sum_{j=1}^{i-1}u_j^{(1)}\right)+b_i=b_i=\mathcal{B}_i\lifting\text{ on }\partial O_i.
\]
Assuming that the equality
\[
\mathcal{B}_i\left(u^{(k-1)}+\sum_{j=1}^iu_j^{(k)}\right)=b_i,
\]
holds for some positive integer $k$ and any integer $i$ in $\left\{1,\dots,N\right\}$, one has, according to identity \eqref{update} and problem~\eqref{iterseq},
\begin{align*}
\mathcal{B}_i\left(u^{(k)}+\sum_{j=1}^iu_j^{(k+1)}\right)&=
\mathcal{B}_i\left(u^{(k-1)}+\sum_{j=1}^Nu_j^{(k)}+\sum_{j=1}^iu_j^{(k+1)}\right)\\
&=\mathcal{B}_i\left(u^{(k-1)}+\sum_{j=1}^iu_j^{(k)}+\sum_{j=1}^iu_j^{(k+1)}+\sum_{j=i+1}^Nu_j^{(k)}\right)\\
&=\mathcal{B}_i \left(u^{(k-1)}+\sum_{j=1}^iu_j^{(k)}\right).
\end{align*}
Likewise, for the parallel form, one has, for any integer $i$ in $\left\{1,\dots,N\right\}$,
\[
\mathcal{B}_i\left(u^{(0)}+u_i^{(1)}\right)=\mathcal{B}_iu^{(0)}+b_i-\mathcal{B}_iu^{(0)}=b_i%=\mathcal{B}_i\lifting\text{ on }\partial O_i
,
\]
using the linearity of the operators, problem \eqref{initpar}, the definition of $\lifting$, and assuming that
\[
\mathcal{B}_i\left(u^{(k-1)}+u_i^{(k)}\right)=\mathcal{B}_i\lifting%\text{ on }\partial O_i
\]
is satisfied for some integer positive $k$ and any integer $i$ in $\left\{1,\dots,N\right\}$, one finds that  
\begin{align*}
\mathcal{B}_i\left(u^{(k)}+u_i^{(k+1)}\right)&=\mathcal{B}_i\left(u^{(k-1)}+\sum_{j=1}^Nu_j^{(k)}+u_i^{(k+1)}\right)\\
&=\mathcal{B}_i\left(u^{(k-1)}+\sum_{j=1}^Nu_j^{(k)}-\sum_{j=1}^{i-1}u_j^{(k)}-\sum_{j=i+1}^Nu_j^{(k)}\right)\\
&=\mathcal{B}_i\left(u^{(k-1)}+u_i^{(k)}\right),
\end{align*}
using identity \eqref{update} and problem \eqref{iterpar}, which ends the proof.
\end{proof}

\begin{remark}
An analogue of formula \eqref{correction formula parallel} was derived by H\"ofer and Vel\'azquez (see equation (28) in~\cite{Hofer:2018}) in the case of the screened Poisson equation with homogeneous Dirichlet boundary conditions and an \emph{infinite} number of objects.
\end{remark}

To get a result on the possible limit of sequence~\eqref{approximate solution}, a minimal functional framework needs to be specified. Namely, we assume from now on that problem~\eqref{BVP main equation}-\eqref{BVP additional condition} is set in a Banach space.

\begin{proposition}\label{prop:banach}
If the sequence of functions, defined by \eqref{approximate solution} and produced by either form of the method of reflections, converges (in the sense of an absolutely convergent series), then it is to the solution $u^*$ to boundary value problem \eqref{BVP main equation}-\eqref{BVP additional condition}.
\end{proposition}
\begin{proof}
From problems \eqref{definition u0}, \eqref{initseq} and \eqref{iterseq} (resp. \eqref{initpar} and \eqref{iterpar}), and formula \eqref{approximate solution}, it is easily seen that any element of the sequence $(u^{(k)})_{k\in\mathbb{N}}$, produced by the sequential (resp. parallel) form of the method of reflections, satisfies
\[
\left\{\begin{aligned}
&\mathcal{L}u^{(k)}=f\text{ in }\Omega\setminus\cup_{j=1}^N\overline{O_j},\\
&\mathcal{B}_0u^{(k)}=0,
\end{aligned}\right.
\]
and thus, by continuity, so does the limit of this sequence. Next, using that the convergence implies that
\[
\forall i\in\{1,\dots,N\},\ \lim_{k\to+\infty}u^{(k)}_i=0,
\]
it follows from passing to the limit in $k$ in \eqref{correction formula sequential} (resp. \eqref{correction formula parallel}) that
\[
\forall i\in\{1,\dots,N\},\ \lim_{k\to+\infty}\mathcal{B}_iu^{(k)}=b_i,
\]
so that the limit of the sequence indeed solves problem \eqref{BVP main equation}-\eqref{BVP additional condition}.
\end{proof}

\begin{remark}
For any integer $i$ in $\{1,\dots,N\}$, the partial sum $\sum_{\ell=0}^k{u_i^{(\ell)}}$ represents the contribution of the $i$th object to the approximate solution after $k$ cycles. As the convergence of the method entails the convergence of the sequences of partial sums, the method of reflections leads to a constructive and natural (\textit{i.e.}, with respect to the object boundaries) decomposition result for the solution to the boundary value problem, similar to the one obtained in Theorem 1 in \cite{Balabane:2004}. The existence of such a decomposition, which can be viewed as a specific application of the superposition principle, suggests that the method is best combined with numerical methods based on an integral representation formula. The basic ingredient in this formalism is the so-called Green function of the governing linear differential operator, which also allows to compute the initial approximation as a volume potential (also called a Newton potential in the case of the Laplace equation) and the subsequent reflections as surface (single-layer and/or double-layer) potentials.
\end{remark}

\subsection{Projection setting}\label{projection setting}
Following Luke \cite{Luke:1989} and H\"ofer and Vel\'azquez \cite{Hofer:2018,Hoefer:2018}, we propose an interpretation of the method of reflections in terms of projection operators. In order to properly define such operators, the problem to be solved and the various subproblems considered by the method, all defined on different subsets of $\Omega$, first need to be extended to the whole of $\Omega$, by the adjunction of problems defined in the interior of the objects and the use of transmission conditions across the boundaries of the objects. In what follows, we will abuse the notation by denoting in the same way the functions defined on $\Omega\setminus\cup_{j=1}^N\overline{O_j}$ (the reflection $u_i^{(k)}$, the approximation $u^{(k)}$ and the solution $u^*$)  and their respective extensions.

Let $H$ be a function space defined over the set $\Omega$ in which the resulting exterior-interior transmission problems and the initialisation problem \eqref{definition u0} are well-posed. For any function $u$ in $H$, consider the functions $v$ and $w$ in $H$, such that $u=v+w$, with $w$ satisfying the following problem in the exterior of the objects,
\begin{equation}\label{exterior problem for w}
\left\{\begin{aligned}
&\mathcal{L}w=\mathcal{L}u\text{ in }\Omega\setminus\cup_{j=1}^N\overline{O_j},\\
&\mathcal{B}_jw=0%\text{ on }\partial O_j
,\ j=1,\dots,N,\\
&\mathcal{B}_0w=0,
\end{aligned}\right.
\end{equation}
the following equation in the interior of the objects,
\begin{equation}\label{interior problem for w}
\mathcal{L}w=0\text{ in }\cup_{j=1}^NO_j,
\end{equation}
and some chosen transmission conditions across the boundaries of the objects, which depend on the differential operator $\mathcal{L}$ and are such that a solution to the resulting exterior-interior transmission problem \eqref{exterior problem for w}-\eqref{interior problem for w} exists and is uniquely defined in $H$.
%whose restrictions to $\Omega\setminus\cup_{j=1}^N\overline{O_j}$ are the respective solutions to
%\[
%\left\{\begin{aligned}
%&\mathcal{L}v=\mathcal{L}u\text{ in }\Omega\setminus\cup_{j=1}^N\overline{O_j},\\
%&\mathcal{B}_jv=0\text{ on }\partial O_j,\ j=1,\dots,N,\\
%&\mathcal{B}_0v=0,
%\end{aligned}\right.
%\text{ and }
%\]
The functions $v$ and $w$ defined in this way are unique and depend continuously on $u$. It then follows that the mappings $u\mapsto v$ and $u\mapsto w$ from $H$ to $H$ are bounded. Since they also are idempotent by construction, they are continuous linear projection operators, with respective closed ranges $V$ and $M$ such that
\begin{equation}\label{direct sum decomposition}
H=V\oplus M.
\end{equation}
In what follows, we shall assume that boundary value problem \eqref{BVP main equation}-\eqref{BVP additional condition} admits a lifting of its boundary data, \textit{i.e.}, there exists a function $\lifting$ defined on $\Omega$ such that $\mathcal{B}_ig=b_i$ on $\partial O_i$, for any integer $i$ in $\{0,\dots,N\}$.

\begin{proposition}\label{proposition solution by projections}
Suppose that problem \eqref{BVP main equation}-\eqref{BVP additional condition} admits a lifting $\lifting$ in $H$ of its boundary data, that there exists an initial approximation $u^{(0)}$ in $H$ satisfying \eqref{definition u0}, and that the multiple-object exterior-interior transmission problem \eqref{exterior problem for w}-\eqref{interior problem for w} is well-posed in $H$ for any choice of $u$ in $H$, so that, in particular, decomposition \eqref{direct sum decomposition} holds. Then, the solution to problem \eqref{BVP main equation}-\eqref{BVP additional condition} is given by the restriction to $\Omega\setminus\cup_{j=1}^N\overline{O_j}$ of 
\[
u^*=P_V\lifting+P_Mu^{(0)},
\]
where $P_V$ denotes the projection from $H$ onto $V$ along $M$, and $P_M$ denotes the projection from $H$ onto $M$ along $V$.
\end{proposition}
\begin{proof}
Owing to the definitions of $\lifting$ and $u^{(0)}$, one can rewrite problem \eqref{BVP main equation}-\eqref{BVP additional condition} as
\[
\left\{\begin{aligned}
&\mathcal{L}u=\mathcal{L}u^{(0)}\text{ in }\Omega\setminus\cup_{j=1}^N\overline{O_j},\\
&\mathcal{B}_ju=\mathcal{B}_j\lifting%\text{ on }\partial O_j
,\ j=1,\dots,N,\\
&\mathcal{B}_0u=\mathcal{B}_0\lifting.
\end{aligned}\right.
\]
%Denoting respectively by $v$ and $w$ the restrictions to $\Omega\setminus\cup_{j=1}^N\overline{O_j}$ of $P_V\lifting$ and $P_Mu^{(0)}$, we observe that $v$ and $w$ satisfy the following systems in the exterior of the objects
%\[
%\left\{\begin{aligned}
%&\mathcal{L}v=0\text{ in }\Omega\setminus\cup_{j=1}^N\overline{O_j},\\
%&\mathcal{B}_jv=\mathcal{B}_j\lifting\text{ on }\partial O_j,\ j=1,\dots,N,\\
%&\mathcal{B}_0v=\mathcal{B}_0\lifting,
%\end{aligned}\right.
%\text{ and }
%\left\{\begin{aligned}
%&\mathcal{L}w=\mathcal{L}u^{(0)}\text{ in }\Omega\setminus\cup_{j=1}^N\overline{O_j},\\
%&\mathcal{B}_jw=0\text{ on }\partial O_j,\ j=1,\dots,N,\\
%&\mathcal{B}_0w=0,
%\end{aligned}\right.
%\]
Observing that $P_V\lifting$ and $P_Mu^{(0)}$ satisfy the following systems in the exterior of the objects
\[
\left\{\begin{aligned}
&\mathcal{L}(P_V\lifting)=0\text{ in }\Omega\setminus\cup_{j=1}^N\overline{O_j},\\
&\mathcal{B}_j(P_V\lifting)=\mathcal{B}_j\lifting%\text{ on }\partial O_j,
\ j=1,\dots,N,\\
&\mathcal{B}_0(P_V\lifting)=\mathcal{B}_0\lifting,
\end{aligned}\right.
\text{ and }
\left\{\begin{aligned}
&\mathcal{L}(P_Mu^{(0)})=\mathcal{L}u^{(0)}\text{ in }\Omega\setminus\cup_{j=1}^N\overline{O_j},\\
&\mathcal{B}_j(P_Mu^{(0)})=0%\text{ on }\partial O_j
,\ j=1,\dots,N,\\
&\mathcal{B}_0(P_Mu^{(0)})=0,
\end{aligned}\right.
\]
we conclude using the linearity of the problems.
\end{proof}

\medskip

\begin{remark}\label{remark interior problem}
The function $u^*$ introduced in Proposition \ref{proposition solution by projections} is an extension to the whole of domain $\Omega$ of the solution of the problem set in the perforated domain $\Omega\setminus\cup_{j=1}^N\overline{O_j}$, satisfying
\[
\mathcal{L}u^*=\mathcal{L}g\text{ in }\cup_{j=1}^NO_j.
\]
\end{remark}

For any integer $i$ in $\{1,\dots,N\}$, considering well-posed single-object exterior-interior transmission problems in place of \eqref{exterior problem for w}-\eqref{interior problem for w}, made up of systems of the form
\begin{equation}\label{problem for w_i}
\forall i\in\{1,\dots,N\},\ \left\{\begin{aligned}
&\mathcal{L}w=\mathcal{L}u\text{ in }\Omega\setminus\overline{O_i},\\
&\mathcal{B}_iw=0%\text{ on }\partial O_i
,\\
&\mathcal{B}_0w=0,
\end{aligned}\right.
\text{ and }\mathcal{L}w=0\text{ in }O_i,
\end{equation}
completed by a transmission condition across the boundary $\partial O_i$ similar to those previously chosen for the multiple-object problem, one can similarly introduce the projection operators $P_{V_i}$ and $P_{M_i}$, with respective complementary closed ranges $V_i$ and $M_i$ in $H$. It is then clear that
\[
M=\cap_{j=1}^N M_j.
\]
As was done in Proposition \ref{proposition solution by projections} for the solution to problem \eqref{BVP main equation}-\eqref{BVP additional condition}, the reflections can be shown to be the restrictions of some quantities involving these projection operators. More precisely, we have the following result, whose proof is left to the reader.

\begin{proposition}\label{proposition reflection by projections}
Suppose that problem \eqref{BVP main equation}-\eqref{BVP additional condition} admits a lifting $\lifting$ in $H$ of its boundary data, that there exists an initial approximation $u^{(0)}$ in $H$ satisfying \eqref{definition u0}, and that the single-object exterior-interior transmission problems \eqref{problem for w_i} are well-posed in $H$ for any choice of $u$ in $H$. Then, for any integer $i$ in $\{1,\dots,N\}$, the reflection $u_i^{(k)}$ is given by the restriction to $\Omega\setminus\overline{O_i}$ of 
\[
-P_{V_i}\left(u^{(0)}+\sum_{j=1}^{i-1}u_j^{(1)}-\lifting\right)\text{(resp. } -P_{V_i}\left(u^{(0)}-\lifting\right)\text{) if }k=1,
\]
\[
-P_{V_i}\left(\sum_{j=1}^{i-1}u_j^{(k)}+\sum_{j=i+1}^{N}u_j^{(k-1)}\right)\text{(resp. } -P_{V_i}\left(\sum_{j=1}^{i-1}u_j^{(k-1)}+\sum_{j=i+1}^{N}u_j^{(k-1)}\right)\text{) if }k>1,
\]
for the sequential (resp. parallel) version of the method of reflections, where $P_{V_i}$ denotes the projection from $H$ onto $V_i$ along $M_i$.
\end{proposition}

\medskip

We finally derive some relations needed for the forthcoming analysis of the method of reflections, starting with the sequential version.

\begin{proposition}\label{MR1}
Under the assumptions and notations of Proposition \ref{proposition reflection by projections}, the following recurrence formula holds for the sequence of approximations produced by the sequential form of the method of reflections,
\begin{equation}\label{errformseq}
\forall k\in\mathbb{N},\ u^{(k+1)}-\lifting={E_N}\left(u^{(k)}-\lifting\right),
\end{equation}
where $E_N=(Id_H-P_{V_N})\dots(Id_H-P_{V_1})$, with $Id_H$ the identity operator on $H$.
\end{proposition}
\begin{proof}
Let us denote $E_0=Id_H$, $E_i=(Id_H-P_{V_i})E_{i-1}$, with $i$ in $\{1,\dots,N\}$. We shall first prove by induction that
\[
\forall k\in\mathbb{N},\ \forall i\in\left\{1,\dots,N\right\},\ u^{(k+1)}_i=-P_{V_i}E_{i-1}\left(u^{(k)}-g\right).
\]
First, it follows from \eqref{correction formula sequential} that
\[
\forall k\in\mathbb{N},\ \forall i\in\left\{1,\dots,N\right\},\ \mathcal{B}_i\left(u^{(k+1)}_i\right)=-\mathcal{B}_i\left(u^{(k)}-\lifting+\sum_{j=1}^{i-1}u^{(k+1)}_j\right),
\]
which, due to the definition of the projection operator $P_{V_i}$ and to the extension of problem \eqref{iterseq}, translates into
\[
\forall k\in\mathbb{N},\ \forall i\in\left\{1,\dots,N\right\},\ u^{(k+1)}_i=-P_{V_i}\left(u^{(k)}-\lifting+\sum_{j=1}^{i-1}u^{(k+1)}_j\right).
\]
For $i=1$, this identity is simply
\[
\forall k\in\mathbb{N},\ u^{(k+1)}_1=-P_{V_1}\left(u^{(k)}-\lifting\right).
\]
Next, assume that
\[
\forall k\in\mathbb{N},\ \forall i\in\{1,\dots,N-1\},\ u^{(k+1)}_i=-P_{V_i}E_{i-1}\left(u^{(k)}-\lifting\right).
\]
Then, one has
\[
u^{(k+1)}_{i+1}=-P_{V_{i+1}}\left(u^{(k)}-\lifting+\sum_{j=1}^{i}u^{(k+1)}_j\right)=-P_{V_{i+1}}\left(Id-\sum_{j=1}^{i}P_{V_j}E_{j-1}\right)\left(u^{(k)}-\lifting\right)=-P_{V_{i+1}}E_i\left(u^{(k)}-\lifting\right),
\]
which ends the induction. Using definition \eqref{update}, one finally reaches
\[
u^{(k+1)}-\lifting=u^{(k)}-g+\sum_{i=1}^Nu^{(k+1)}_i=\left(Id-\sum_{i=1}^NP_{V_i}E_{i-1}\right)\left(u^{(k)}-\lifting\right)=E_N\left(u^{(k)}-\lifting\right).
\]
\end{proof}

%\begin{remark}
%For the particular choice $\lifting=u^*$, relation \eqref{errformseq} becomes an error propagation formula characteristic (see Subsection 2.1 in \cite{Xu:2002}) of the so-called \emph{successive subspace correction (SSC) methods}, a notion introduced by Xu \cite{Xu:1992,Xu:2001} to regroup diverse iterative procedures used for solving linear systems of equations and which are similar in nature, like the Gauss-Seidel method or the alternating Schwarz method.
%\end{remark}

\medskip

An analogous result holds for the parallel version of the method.

\begin{proposition}\label{MR2}
Under the assumptions and notations of Proposition \ref{proposition reflection by projections}, the following recurrence formula holds for the sequence of approximations produced by the parallel form of the method of reflections,
\begin{equation}\label{error propagation formula parallel}
\forall k\in\mathbb{N},\ u^{(k+1)}-\lifting=\left(Id-\sum_{i=1}^NP_{V_i}\right)\left(u^{(k)}-\lifting\right).
\end{equation}
\end{proposition}
\begin{proof}
Owing to \eqref{correction formula parallel}, one has
\[
\forall k\in\mathbb{N},\ \forall i\in\{1,\dots,N\},\ \mathcal{B}_i{u^{(k+1)}_i}=-\mathcal{B}_i\left(u^{(k)}-\lifting\right),
\]
which gives, due to the definition of the projection operator $P_{V_i}$ and the extension of \eqref{iterpar},
\[
\forall k\in\mathbb{N},\ \forall i\in\{1,\dots,N\},\ u^{(k+1)}_i=-P_{V_i}\left(u^{(k)}-\lifting\right)%=P_{V_i}\left(\sum_{j=1}^{i-1}u_j^{(k)}+\sum_{j=i+1}^Nu_j^{(k)}\right)
.
\]
Substituting into \eqref{update}, one reaches
\[
\forall k\in\mathbb{N},\ u^{(k+1)}-\lifting=u^{(k)}-\lifting+\sum_{i=1}^Nu^{(k+1)}_i=u^{(k)}-\lifting-\sum_{i=1}^NP_{V_i}\left(u^{(k)}-\lifting\right),
\]
yielding \eqref{error propagation formula parallel}.
%\begin{equation}\label{recurrence relation parallel method}
%\forall k\in\mathbb{N},\ u^{(k+1)}-u=\left(Id-\sum_{i=1}^NP_{V_i}\right)\left(u^{(k)}-u\right),
%\end{equation}
%and thus, by a recursive application of the above identity,
%\[
%\forall k\in\mathbb{N},\ u^{(k)}-u=\left(Id-\sum_{i=1}^NP_{V_i}\right)^{k-1}\left(u^{(0)}-u\right).
%\]
%We conclude by using that, according to \eqref{definition u0} and \eqref{initpar}, $u^{(0)}=\left(\sum_{i=1}^NP_{V_i}\right)u$.
\end{proof}

%\medskip
%
%\begin{remark}
%When $\lifting=u^*$, relation \eqref{error propagation formula parallel} is an error formula typical of the \emph{parallel subspace correction (PSC) methods} (see Subsection 2.2 in \cite{Xu:2002}), a class of iterative methods to which belong the Jacobi method, the Richardson method, the additive Schwarz method or some multigrid methods.
%\end{remark}

\section{Convergence analysis in a Hilbert space setting}\label{section: convergence analysis}
We now flesh out the structure of the previous mathematical setting by making a number of additional assumptions. We suppose that the function space $H$ is a real Hilbert space %, equipped with the inner product $(\cdot,\cdot)_H$ and the corresponding induced norm $\|\cdot\|_H$. Let $H'$ be the dual space of $H$ with respect to the duality pairing $\left<\cdot,\cdot\right>_{H',H}$
and that the problem to solve can be written under an equivalent weak formulation%: given $f$ in $H'$, find $u\in H$ such that
%\begin{equation}\label{variational formulation}
%\forall v\in H,\ a(u,v)=\left<f,v\right>_{H',H},
%\end{equation}
%where 
, involving a continuous bilinear form $a(\cdot,\cdot)$ from $H\times H$ to $\mathbb{R}$ %, induced by the operator $\mathcal{L}$
satisfying inf-sup conditions which ensure the well-posedness of the problem. In this context, the projection operators $P_{M_i}$, with $i$ in $\{1,\dots,N\}$, previously introduced can be defined by
\begin{equation}\label{definition of projections in Hilbert context}
\forall u\in H,\ \forall v\in M_i,\ a(P_{M_i}u,v)=a(u,v),
\end{equation} % A VOIR In this context, the space $V$ is the closed subspace of $H$ whose elements are (weak) solutions of the problem %\eqref{variational formulation}
% when the source term is associated to the whole set of objects, while the subspace $V_i$, $i=1,\dots,N$, is the set of solutions for which the source term is only related to the $i$th object. It is clear that the set $V_i$ is a closed subspace of $V$, and we naturally have, owing to the nature of the subproblems, that $V=\overline{\sum_{i=1}^NV_i}$ (\textbf{CETTE SOMME EST-ELLE TOUJOURS FERMEE ?}).
%
%Let us recall some definitions taken from \cite{Deutsch:2010}. Assume that $X$ and $Y$ are normed linear spaces over the same field and let $\mathcal{L}(X,Y)$ denote the normed linear space of all bounded linear operators from $X$ to $Y$, endowed with the usual operator norm.
%
%\begin{definition}
%A sequence $(L^{(k)})_{k\in\mathbb{N}}$ in $\mathcal{L}(X,Y)$ is said to converge to an operator $L$ in $\mathcal{L}(X,Y)$ \textbf{in norm} (resp., \textbf{pointwise}) if
%\[
%\lim_{k\to+\infty}\norm{L^{(k)}-L}_{\mathcal{L}(X,Y)}=0\text{ (resp., }\lim_{k\to+\infty}\norm{L^{(k)}x-Lx}_Y=0,\ \forall x\in X\text{).}
%\]
%\end{definition}
 (see \cite{Xu:2002} for instance), and the method of reflections is said to be \emph{convergent in $H$} if the sequence $(u^{(k)})_{k\in\mathbb{N}}$ tends to $u^*$ with respect to the norm on $H$. In view of the respective formulas derived in Propositions \ref{MR1} and \ref{MR2}, it is then clear that the convergence of the method is tied to the behaviour of the fixed-point iteration of a bounded linear mapping $T$ from $H$ to itself, where $T=E_N=(Id_H-P_{V_N})\dots(Id_H-P_{V_1})$ for the sequential \textcolor{black}{form of the} method of reflections, or $T=Id_H-\sum_{i=1}^NP_{V_i}$ for the parallel one.

\subsection{The orthogonal case: theoretical results} %\label{ortho:theory}
For some boundary value problems of the form \eqref{BVP main equation}-\eqref{BVP additional condition}, notably when the operator $\mathcal{L}$ is of elliptic type and the boundary conditions are essential, convenient choices of the Hilbert space $H$ make the bilinear form $a(\cdot,\cdot)$ introduced above an equivalent inner product on $H$. This results in the linear mappings $P_{M_i}$, with $i$ in $\{1,\dots,N\}$, satisfying \eqref{definition of projections in Hilbert context}, being \emph{orthogonal projection operators}. This observation is at the origin of the \emph{method of orthogonal projections}, developed by Vishik \cite{Vishik:1949} in connection with previous works by Zaremba \cite{Zaremba:1927} and Weyl \cite{Weyl:1940}, and adapted by Hruslov for a problem for an elliptic operator set in a perforated domain \cite{Hruslov:1972}.

Under the above assumptions, for any integer $i$ in $\{1,\dots,N\}$, the subspace $V_i={M_i}^\perp$ is the \emph{orthogonal complement of $M_i$ in $H$} and $P_{V_i}=Id-P_{M_i}$ is also an orthogonal projection operator. Likewise, we have that $V=M^\perp$, so that
\[ %begin{equation}
V=\left(\cap_{j=1}^NM_j\right)^\perp=\overline{\sum_{j=1}^N{M_j}^\perp}=\overline{\sum_{j=1}^NV_j}.
\] %end{equation}

We will now address the convergence of both forms of the method within this framework.

\subsubsection{Sequential form}
In the case of orthogonal projection operators, the sequential form of the method of reflections is closely related to the \emph{method of alternating (or cyclic) projections} (MAP for short), which is a simple iterative procedure for determining the orthogonal projection of an element %in $H$ onto $\cap_{i=1}^NM_i$ by cyclically computing the orthogonal projection of a current iterate onto the individual spaces $M_i$, $i=1,\dots,N$.
onto an intersection of closed (linear) subspaces $M_1,\dots,M_N$ of a Hilbert space $H$ using a sequence of orthogonal projections onto those subspaces. It has numerous applications (the interested reader may check the review by Deutsch \cite{Deutsch:1992} for details) and its %This method has applications for the solution of linear systems (as the Kaczmarz method) or partial differential equations (as the Schwarz method), for the approximation of multivariate functions, for the numerical construction of conformal mappings, in probability and statistics (linear prediction theory, restricted least-squares regression), in multigrid methods, in image restoration and in computed tomography (the interested reader will find many more details in the review by Deutsch \cite{Deutsch:1992}). Its very 
design relies on the pointwise convergence result%\footnote{Note that this result remains valid when the closed subspaces are replaced by closed affine sets having a nonempty intersection (see Section 4 in \cite{Deutsch:1997}).}
, first proved by von Neumann in 1933 (but not published until 1949) in the $N=2$ case while working on the theory of operators \cite{vonNeumann:1949} and later generalised by Halperin~\cite{Halperin:1962} to any integer $N\geq2$% (an extension to closed convex sets is due to Bregman \cite{Bregman:1965})
, stating that one has
\[
\forall v\in H,\ \lim_{k\to+\infty}\|(P_{M_N}P_{M_{N-1}}\dots P_{M_1})^kv-P_{\cap_{j=1}^NM_j}v\|_{H}=0.
\]
%Thus, the sequence of operators $((P_{M_N}P_{M_{N-1}}\dots P_{M_1})^k)_{k\in\mathbb{N}}$ converges pointwise to $P_{\cap_{i=1}^NM_i}$. 
%This gives rise to the algorithm defining the MAP, which consists of, for any $v$ in $H$, building a sequence $(v^{(k)})_{k\in\mathbb{N}}$ defined by
%\[
%w^{(0)}=v\text{ and},\ \forall k\in\mathbb{N},\ w^{(k+1)}=P_N\dots P_1w^{(k)},
%\]
%that is, by recursively applying the orthogonal projections to the current iterate.
Consequently, the computation of the orthogonal projection of a point $v$ in $H$ onto $\cap_{j=1}^NM_j$ by the MAP consists in building the sequence $(w^{(k)})_{k\in\mathbb{N}}$ defined by
\[
w^{(0)}=v\text{ and},\ \forall k\in\mathbb{N},\ w^{(k+1)}=P_{M_N}\dots P_{M_1}w^{(k)},
\]
that is, by recursively applying the orthogonal projections to the current iterate. It follows that recurrence relation \eqref{errformseq} for the sequential version of the method of reflections is simply a particular application of the MAP. Thus, using that $P_V+P_M=Id_H$ and that, for any integer $i$ in $\{1,\dots,N\}$, $P_{V_i}P_M=0$, one can write that
\[
\forall k\in\mathbb{N}^*,\ u^{(k)}=(E_N)^k(u^{(0)}-P_V\lifting)+P_V\lifting,
\]
with $E_N=(Id_H-P_{V_N})\dots(Id_H-P_{V_1})=P_{M_N}P_{M_{N-1}}\dots P_{M_1}$, so that passing to the limit yields
\[
\lim_{k\to+\infty}u^{(k)}=P_M(u^{(0)}-P_V\lifting)+P_V\lifting=P_M(u^{(0)})+P_V\lifting=u^*\text{ in }H.
\]
The unconditional convergence of the method ensues and we have proved the following result.
%\smallskip

% Owing to Proposition \ref{MR1}, one may easily infer that the convergence of follows from the convergence of the MAP.% Let us denote by $P$ the projector associated with the initial problem~\eqref{BVP}. It is a simple matter to check that if this problem is well-posed,
%$$\Ker P=\cap_{i=1}^N\Ker P_{V_i}.$$

\begin{theorem}\label{convergence sequential}
Under the assumptions and notations of Proposition~\ref{proposition reflection by projections}, suppose that the projection operators $P_{V_i}$, with $i$ in $\{1,\dots,N\}$, are orthogonal. Then, the sequence of approximations $\left(u^{(k)}\right)_{k\in\mathbb{N}}$ generated by the sequential form of the method of reflections converges to the solution $u^*$ in $H$.
\end{theorem}
 
%Recall that in an Hilbert space, the continuity of projection operators is equivalent to their orthogonality.
%\begin{proof}
%Since the projection operators $(P_{V_i})_{i=1,\dots,N}$ are orthogonal and thus continuous, their kernels are closed, so that Theorem~\ref{convergence result for the MAP} applies with $M_i=\Ker P_{V_i}$. As a consequence, the sequence $\left(\pi_N^{k}u\right)_{k\in\mathbb{N}}$ converges to $P_{\cap_{1\leq i\leq N}\Ker P_{V_i}}u$. As a solution of~\eqref{BVP}, $u$ belongs to $\Im P$. Because of the orthogonality of $P$, one has 
%\[
%\Im P = \left(\Ker P \right)^{\perp}=\left(\cap_{i=1}^N\Ker P_{V_i}\right)^{\perp},
%\]
%so that $P_{\cap_{1\leq i\leq N}\Ker P_{V_i}}u=0$.

%$\cap_{1\leq i\leq N}\left( Im(P_{V_i})\right)=\cap_{1\leq i\leq N}\left( \Ker(P_{V_i})\right)^{\perp}\subset \left( \cap_{1\leq i\leq N}\Ker(P_{V_i}) \right)^{\perp}$ so that because of Proposition~\ref{MR1}, the sequence $\left(u^k \right)_{k\in\mathbb{N}}$ converges to $u$. 
%\end{proof}

\subsubsection{Parallel form} %\label{subsection: parallel form}
The fact that the projection operators are orthogonal does not allow one to conclude that the parallel version of the method of reflections is convergent as it is the case for its sequential form. Nevertheless, if one is able to show that the operator $T=Id_H-\sum_{j=1}^NP_{V_j}$ is nonexpansive\footnote{The mapping $T$ is said to be \emph{nonexpansive} on the normed space $H$ if it is a function from $H$ to itself such that \[\forall(u,v)\in H^2,\ \norm{Tu-Tv}_H\leq\norm{u-v}_H.\]} and  asymptotically regular\footnote{The mapping $T$ is said to be \emph{asymptotically regular} on the normed space $H$ if it is a function from $H$ to itself such that \[\forall u\in H,\ \lim_{k\to+\infty}\norm{T^{k+1}u-T^ku}_H=0.\]}, it is known (see Corollary 2.3 in \cite{Bauschke:2003}) that, for any $u$ in $H$, the sequence $(T^ku)_{k\in\mathbb{N}}$ converges to $P_{\operatorname{Fix}T}u$ in $H$, where $P_{\operatorname{Fix}T}$ is the orthogonal projector on the fixed point set for $T$, and that\footnote{Indeed, one has trivially that $M\subset\operatorname{Fix}T$, while the converse follows from the idempotence and the self-adjointness of the orthogonal projectors.} $\operatorname{Fix}T=M$. Two major drawbacks of such a result is that the imposed conditions on the operator $T$ are often not readily checkable and that it does not hold in a nonorthogonal context. As a consequence, geometric conditions on the objects and their respective positions ensuring the summability of the series defined by \eqref{approximate solution} are generally preferred (see for instance \cite{Traytak:2006} or \cite{Balabane:2004} in a nonorthogonal framework).%(one may nevertheless consult the article \cite{Bjorstad:1991} for some results on the spectrum of a sum of orthogonal projection operators). % A VOIR \emph{However, one may trivially observe that error formula \eqref{error propagation formula parallel} is that of a stationary Richardson iteration\footnote{Indeed, in its simplest form (that is, without preconditioning), the stationary Richardson iteration for the solution of the linear system $Ax=b$ is defined by the following recurrence relation \[\forall k\in\mathbb{N},\ x^{(k+1)}=(I-A)\,x^{(k)}+b,\] the initialisation $x^{(0)}$ being arbitrarily chosen.} type method applied to the solution of the linear system
%\[
%\left(\sum_{i=1}^NP_{V_i}\right)u^*=u^{(0)},
%\]
%involving the sometimes called \emph{additive Schwarz operator} $\sum_{i=1}^NP_i$. It then appears that the solution $u$ may be advantageously obtained by a Krylov space method like the conjugate gradient method (when the projection operators are symmetric), which converges with a speed determined by the condition number of $\sum_{i=1}^NP_i$. Such an approach is used in \cite{Ichiki:2001} to solve a resistance problem for which the method of reflections breaks down.}

Nevertheless, one maay obtain an unconditionally convergent and properly parallel method by modifying the algorithm in such a way that the resulting recurrence formula reads%\footnote{Note that any other convex combination of the operators $(Id-P_{V_1}),\dots,(Id-P_{V_N})$ could alternatively be used.}
\[
\forall k\in\mathbb{N},\ u^{(k+1)}-\lifting=\left(Id-\frac{1}{N}\,\sum_{j=1}^NP_{V_j}\right)\left(u^{(k)}-\lifting\right)=\left(\frac{1}{N}\,\sum_{j=1}^N(Id-P_{V_j})\right)\left(u^{(k)}-\lifting\right).
\]
This alteration can be interpreted as a \emph{relaxation} of the recurrence relation \eqref{error propagation formula parallel} of the form
\[
\forall k\in\mathbb{N},\ u^{(k+1)}-\lifting=\nu\left(Id-\sum_{j=1}^NP_{V_j}\right)\left(u^{(k)}-\lifting\right)+(1-\nu)\,\left(u^{(k)}-\lifting\right),
\]
with relaxation factor $\nu=\frac{1}{N}$, yielding
\[
\forall k\in\mathbb{N},\ u^{(k)}-\lifting=\left(\frac{1}{N}\,\sum_{j=1}^N(Id-P_{V_j})\right)^k\left(u^{(0)}-\lifting\right).
\]
It is then a well-known fact, since the work of Cimmino \cite{Cimmino:1938} on the \emph{method of averaged (or simultaneous) projections} (an iterative method to solve linear systems of equation based on a geometrical approach), and its extension by Auslender \cite{Auslender:1976}, that, for any closed (linear) subspaces $M_1,\dots,M_N$ of a Hilbert space $H$, one has the pointwise convergence result\footnote{This result remains valid for the more general sum $\sum_{j=1}^N\gamma_j\,P_{M_j}$, where the scalars $\gamma_i$, with $i$ in $\{1,\dots,N\}$, are the weights of a convex combination (that is, such that, $\forall i\in\{1,\dots,N\}$, $\gamma_i>0$, and $\sum_{j=1}^n\gamma_j=1$).}
\[
\forall v\in H,\ \lim_{k\to+\infty}\left\|\left(\frac{1}{N}\,\sum_{j=1}^NP_{M_j}\right)^kv-P_{\cap_{j=1}^NM_j}v\right\|_{H}=0.
\]
The convergence of the method then follows from the same arguments as in the previous subsection.
 
%By choosing $M_i=V_i^\perp$, $i=1,\dots,N$, and noting that $\cap_{i=1}^NM_i=\cap_{i=1}^NV_i^\perp=(\sum_{i=1}^NV_i)^\perp=\{0_H\}$ owing to assumption \eqref{decomposition of H}, we have obtained the following convergence result.

\begin{theorem}\label{convergence averaged parallel}
Under the assumptions and notations of Proposition~\ref{proposition reflection by projections}, suppose that the projection operators $P_{V_i}$, with $i$ in $\{1,\dots,N\}$, are orthogonal. Then, the sequence of approximations $\left(u^{(k)}\right)_{k\in\mathbb{N}}$ generated by the averaged  parallel version of the parallel form of the method of reflections converges to the solution $u^*$ in $H$.
\end{theorem}

In terms of induced practical changes, the update formula for the approximation of the solution (see equation \eqref{update}) becomes
\[
\forall k\in\mathbb{N},\ u^{(k+1)}=u^{(k)}+\frac{1}{N}\,\sum_{j=1}^Nu^{(k+1)}_j\text{ in }\Omega\setminus\cup_{j=1}^N\overline{O_j},
\]
the reflection $u^{(k+1)}_i$ being obtained by solving sub-problem \eqref{iterpar} in which the second equation is replaced by
\[
\mathcal{B}_i(u^{(k+1)}_i)=-\mathcal{B}_i\left(\frac{1}{N}\left(\sum_{j=1}^{i-1}u_j^{(k)}+\sum_{j=i+1}^Nu_j^{(k)}\right)-\left(1-\frac{1}{N}\right)u_i^{(k)}\right)%\text{ on }\partial O_i
,
\]
confirming that the modification may be viewed as a simple relaxation method applied to the parallel form of the method of reflections. One may also remark that the interpretation of the auxiliary fields as ``reflections'' remains, since it can be shown that equality \eqref{correction formula parallel} still holds. We thus refer to this variant as the \emph{averaged} parallel form of the method of reflections.

\subsubsection{Rate of convergence}
Let us now deal with the rate of convergence of the method by recalling the following dichotomy property for the MAP (see Theorem 6.4 in \cite{Deutsch:2010} and also Theorem 1.4 in \cite{Bauschke:2009} for the case $N=2$):

%\begin{theorem}\label{dichotomy theorem}
%Let $M_1,\dots,M_N$ be closed subspaces of a Hilbert space H. Then exactly one of the following two statements holds.
\begin{itemize}
\item If $\sum_{j=1}^N{M_j}^\perp$ is closed, the sequence $((P_{M_N}P_{M_{N-1}}\dots P_{M_1})^k)_{k\in\mathbb{N}}$ converges to $P_{\cap_{j=1}^NM_j}$ linearly\footnote{Some authors say the sequence converges \emph{uniformly} (see \cite{Badea:2012}).}, that is, there exist constants $C>0$ and $\alpha\in[0,1)$ such that
\[
\forall k\in\mathbb{N},\ \norm{(P_{M_N}P_{M_{N-1}}\dots P_{M_1})^k-P_{\cap_{j=1}^NM_j}}_{\mathscr{L}(H)}\leq C\,\alpha^k.
\]
\item If $\sum_{j=1}^N{M_j}^\perp$ is not closed, the sequence $((P_{M_N}P_{M_{N-1}}\dots P_{M_1})^k)_{k\in\mathbb{N}}$ converges to $P_{\cap_{j=1}^NM_j}$ arbitrarily slowly, that is,
\begin{enumerate}[(i)]
\item the sequence $((P_{M_N}P_{M_{N-1}}\dots P_{M_1})^k)_{k\in\mathbb{N}}$ converges pointwise to $P_{\cap_{j=1}^NM_j}$,
\item for each real-valued function $\phi$ on the positive integers that converges to $0$, there exists a point $v$ in $H$ such that
\[
\forall k\in\mathbb{N},\ \norm{(P_{M_N}P_{M_{N-1}}\dots P_{M_1})^kv-P_{\cap_{j=1}^NM_j}v}_H\geq\phi(k).
\]
\end{enumerate}
\end{itemize}
%\end{theorem}

\medskip

The implication of this result for the sequential form of the method of reflections is that it converges linearly as soon as the sum $\sum_{j=1}^NV_j$ is closed, that is if $V=\sum_{j=1}^NV_j$. Equivalent conditions, based on the notion of angle\footnote{According to the definition of Friedrichs \cite{Friedrichs:1937}, if $M_i$ and $M_j$ are closed subspaces in a Hilbert space $H$, the angle between $M_i$ and $M_j$ is the angle in $[0,\frac{\pi}{2}]$ whose cosine is defined by \[c(M_i,M_j)=\sup\left\{\abs{\left<x,y\right>}\,|\,x\in M_i\cap(M_i\cap M_j)^\bot,\ \norm{x}\leq 1,\ y\in M_j\cap(M_i\cap M_j)^\bot,\ \norm{y}\leq1\right\}.\] Another notion is that of the minimal angle between $M_i$ and $M_j$, given by Dixmier in \cite{Dixmier:1949}, which is the angle in $[0,\frac{\pi}{2}]$ whose cosine is defined by \[c_0(M_i,M_j)=\sup\left\{\abs{\left<x,y\right>}\,|\,x\in M_i,\ \norm{x}\leq 1,\ y\in M_j,\ \norm{y}\leq1\right\}.\] These two definitions are different if $M_i\cap M_j\neq\{0\}$, and they both coincide with (different) principal angles (as introduced by Jordan \cite{Jordan:1875}) if $H=\mathbb{C}^n$.} between subspaces, can be found in the literature\footnote{For $N=2$, it is known (see \cite{Deutsch:1985} for instance) that $c(M_1,M_2)<1$, where $c(M_1,M_2)$ denotes the cosine of the angle between the subspaces $M_1$ and $M_2$, if and only if $M_1+M_2$ is closed, if and only if ${M_1}^\bot+{M_2}^\bot$ is closed, if and only if $(M_1\cap(M_1\cap M_2)^\bot)+(M_2\cap(M_1\cap M_2)^\bot)$ is closed. For $N\geq3$, a generalization of the Friedrichs angle to several subspaces is introduced in \cite{Badea:2012} and it is shown in the same paper %(Theorem 4.4)
 that the method converges linearly if and only if $c(M_1,\dots,M_N)<1$, which is a weaker condition (see Example 4.5 in \cite{Badea:2012}) than the sufficient one, based on Theorems 2.1 in \cite{Deutsch:1997} and 4.1 in \cite{Badea:2012}, that one of the cosines $c_{ij}=c_0(M_i\cap(\cap_{\ell=1}^N M_\ell)^\bot,M_j\cap(\cap_{\ell=1}^N M_\ell)^\bot)$ of the Dixmier angles involving two subspaces $M_i$ and $M_j$ is strictly less than one.}, and it is worth noting that the linear convergence of the (sequential version of the) method of reflections was established in \cite{Luke:1989} using conditions on any two sub-collections of the set of involved subspaces and the notion of gap between these.

\smallskip

Note that, in the case of a linear convergence of the MAP, error bounds can be derived, leading to estimates for the sequential for of the method of reflections. Let us gather here some of the existing results on this topic. When $N=2$, one has $\dnorm{(P_{M_2}P_{M_1})^k-P_{M_1\cap M_2}}_{\mathscr{L}(H)}\leq c(M_1,M_2)^{2k-1}$, where $c(M_1,M_2)$ denotes the cosine of the angle between the subspaces $M_1$ and $M_2$. This result is due to Aronszajn (see \cite[Section~12]{Aronszajn:1950}) and has been rediscovered several times. It is also sharp (see \cite[Theorem 2]{Kayalar:1988}). For $N\geq3$, upper bounds were given by Smith, Solomon and Wagner \cite[Theorem 2.2]{Smith:1977}, Kayalar and Weinert \cite[Theorem 3]{Kayalar:1988}, and also Deutsch and Hundal \cite[Theorem 2.7]{Deutsch:1997}. In this case however, note that any error bound depending only on the angles between the various subspaces involved can never be sharp \cite[Example 3.7]{Deutsch:1997}. 
%The Smith--Solomon--Wagner bound is easy to compute, as
%\[E_N(k)\leq c^k\text{ with }c=\left(1-\prod_{i=1}^{N-1}\left(1-c\left(M_i,\cap_{j=i+1}^NM_j\right)\right)\right).\]
%On the contrary, the Kayalar--Weinert bound is involved and generally complicated to compute, as it is the minimum over four constants, two of these involve computing the product of $kN$ factors and the other two the product of $2kN$ factors, where the $i$th of these factors involves computing the sum of $i$ terms each of which is made up of quotients of terms involving the cosines of the angles between the various subspaces (the amount of computation needed for computing the bound for the $k$\textsuperscript{th} iterate increases with $k$). The Deutsch--Hundal bound is the minimum of three constants, each of which involves the product of at most $N$ terms, and each of these terms involves a single cosine of the angle between certain of the subspaces (the amount of computation needed for the bound is fixed and is independent of $k$).
More recently, using the link between the MAP and the SSC methods, Xu and Zikatanov \cite{Xu:2002} obtained, under the assumption that the sum $\sum_{j=1}^N{M_j}^\perp$ is closed, the following \emph{equality}
\[
{\norm{P_{M_N}P_{M_{N-1}}\dots P_{M_1}-P_{\cap_{j=1}^NM_j}}_{\mathscr{L}(H)}}^2=\frac{c_0}{1+c_0},
\]
where
\[
c_0=\underset{\substack{v\in V\\\norm{v}_H=1}}{\sup}\,\underset{\substack{(v_1,\dots,v_N)\in{M_1}^\perp\times\dots\times{M_N}^\perp\\\sum_{k=1}^Nv_k=v}}{\inf}\,\textstyle\sum_{j=1}^N{\norm{(Id_H-P_{M_j})\sum_{k=j+1}^Nv_k}_H}^2.
\]

We may add that various ways of accelerating the MAP, using relaxation or symmetrization for instance, have been proposed and studied (see \cite{Bauschke:2003} and the references therein), and could be directly used on the sequential version of the method of reflections.

\medskip

Finally, because simultaneous projections correspond to alternating projections in the adequate product space (see \cite{Pierra:1984}), variants of the above results exist for the method of simultaneous projections, as a consequence of results in \cite{Bauschke:1996%Lemma 5.18 and Theorem 5.19
,Bauschke:2009%Theorem 1.4
}, and thus apply to the averaged parallel version of the method of reflections.

%Recalling that $P_{{M_i}^\perp}=Id-P_{M_i}$, ${M_i}^{\perp\perp}=M_i$, $\left(\cap_{i=1}^NM_i\right)^\perp=\sum_{i=1}^N{M_i}^\perp$, and setting $V_i={M_i}^\perp$, one finally obtains the following result.
%\begin{corollary}[Corollary 6.7 in \cite{Deutsch:2010}]
%Let $V_1,V_2,\dots,V_N$ be closed subspaces of the Hilbert space H and let $V=\overline{\sum_{i=1}^NV_i}$. Then exactly one of the following two statements holds.
%\begin{enumerate}
%\item $\sum_{i=1}^NV_i$ is closed. Then the sequence $((Id-P_{V_N})(Id-P_{V_{N-1}})\dots (Id-P_{V_1})^k)_{k\in\mathbb{N}}$ converges to $Id-P_V$ linearly.
%\item  $\sum_{i=1}^NV_i$ is not closed. Then the sequence $((Id-P_{V_N})(Id-P_{V_{N-1}})\dots (Id-P_{V_1})^k)_{k\in\mathbb{N}}$ converges to $Id-P_V$ arbitrarily slowly.
%\end{enumerate}
%\end{corollary}

\subsection{The orthogonal case: practical examples}\label{practical cases}
It follows from the previous results that the orthogonality of the projection operators associated with the boundary value problem to solve is a sufficient condition for the convergence of the method, either in its sequential form or in an averaged version of its parallel form. As an application, we prove this property for problems in which the main differential operator is of elliptic type. More precisely, our attention is focused on examples that involve the Laplace and the Stokes operators. For each of them, we describe the extended boundary value problems and characterise the functional setting involved in the definition of the projection operators. In addition to proving the orthogonality of the projectors, we also show that the rate of convergence is linear, by establishing that the sum $\sum_{j=1}^NV_j$ is closed.

\subsubsection{The Laplace and Poisson equations}
In what follows, it is assumed that the domain $\Omega$ is a bounded, simply connected, open set of $\mathbb{R}^d$, with boundary $\partial\Omega$, containing $N$ simply connected open subdomains $O_j$, with respective boundaries $\partial O_j$, with $j$ in $\{1,\dots,N\}$. All the boundaries are supposed to be sufficiently smooth, twice continuously differentiable for instance.

We deal with boundary value problems of the form considered in Section \ref{framework}, for which the operator $\mathcal{L}$ is the negative Laplace operator and $\mathcal{B}_0$ is the trace operator on $\partial\Omega$. These choices respectively correspond to the Poisson equation if the datum $f$ is non-zero, else to the Laplace equation, with a homogeneous Dirichlet boundary condition on $\partial\Omega$. They have to be complemented with some conditions on the boundaries $\partial O_j$, with $j$ in $\{1,\dots,N\}$, of the objects, for which different choices are possible.

\paragraph{\textbf{Poisson problem with Dirichlet boundary conditions.}}
As in the boundary value problem appearing in Subsection~\ref{subsec:introRes}, we consider Dirichlet-type boundary conditions, meaning that, for any integer $i$ in $\{1,\dots,N\}$, the boundary operator $\mathcal{B}_i$ is the trace operator on $\partial O_i$. If the datum $f$ belongs to $H^{-1}(\Omega)$ and the boundary data $b_j$, with $j$ in $\{1,\dots,N\}$, are respectively in $H^{1/2}(\partial O_j)$, one may show, using the Lax--Milgram lemma, that the problem is well-posed in the space $\{u\in H^1(\Omega\setminus\cup_{j=1}^N\overline{O_j})\,|\,u_{|_{\partial\Omega}}=0\}$.

To extend the problem to the interior of the objects, we impose the following transmission conditions across the object boundaries
\begin{equation}\label{vanishing jump conditions, Laplace-Dirichlet}
[u^*]=0\text{ across }\partial O_j,\ j=1,\dots,N,
\end{equation}
where the brackets $[\,\cdot\,]$ denote the jump across the considered hypersurface. We may then set the exterior-interior transmission problem in the space $H=H_0^1(\Omega)$, and it follows from the surjectivity of the trace operator that there exists a lifting $\lifting$ in $H$ such that $\lifting=b_j$ on $\partial O_j$, for any integer $j$ in $\{1,\dots,N\}$. Setting
\[
-\Delta u^*=-\Delta\lifting\text{ in }\cup_{j=1}^NO_j,
\]
according to Remark \ref{remark interior problem}, we see that the solution of the resulting problem satisfies interior Poisson--Dirichlet problems in the objects coupled with an exterior Poisson--Dirichlet problem by the transmission conditions \eqref{vanishing jump conditions, Laplace-Dirichlet}. As a consequence, it exists and is uniquely defined in the space $H$. Finally, the system of equations satisfied by the initial approximation $u^{(0)}$,
\[
\left\{\begin{aligned}
-\Delta u^{(0)}=f\text{ in }\Omega,\\
u^{(0)}=0\text{ on }\partial\Omega,
\end{aligned}\right.
\]
defines a problem admitting a unique weak solution in $H$.

In addition, one may note that solutions to problem \eqref{exterior problem for w}-\eqref{interior problem for w} satisfying \eqref{vanishing jump conditions, Laplace-Dirichlet} vanish in $\cup_{j=1}^N\overline{O_j}$. This observation allows to characterise the Hilbert space framework associated with the method of reflections applied to the Poisson--Dirichlet problem.

\begin{proposition} %\label{diri_case}
Let $H=H_0^1(\Omega)$, equipped with the inner product
\[
a(u,v)=\int_\Omega\grad u(x)\cdot\grad v(x)\,\mathrm{d}x.
\]
For the Laplace operator completed with Dirichlet boundary conditions, the projectors $P_M$ and $P_V$ (resp. $P_{M_{i}}$ and $P_{V_{i}}$, with $i$ in $\{1,\dots,N\}$) introduced in Subsection~\ref{projection setting} are well-defined and orthogonal with respect to $a$, and their respective ranges are characterised as
\[ %\]begin{equation}\label{MV}
M=\left\{w\in H\,|\,w=0\text{ in }\cup_{j=1}^N\overline{O_j}\right\}\text{ and }V=\left\{v\in H\,|\,\Delta v=0\text{ in }\Omega\setminus\cup_{j=1}^N\overline{O_j}\right\}
\] %end{equation}
\[ %\]begin{equation}\label{MiVi}
\text{(resp. }M_i=\left\{w\in H\,|\, w=0\text{ in } \overline{O_i}\right\}\text{ and }V_i=\left\{v\in H\,|\,\Delta v=0\text{ in }\Omega\setminus\overline{O_i}\right\},\ i=1,\dots,N\text{)}.
\] %end{equation}
In addition, one has $V=\sum_{j=1}^NV_j$.
\end{proposition}
\begin{proof}
Let us first deal with the projector $P_M$. For any $u$ in $H$, $w=P_M(u)$ is by definition the weak solution in $H$ to the following exterior-interior transmission problem 
\begin{equation}\label{problem for w, Laplace-Dirichlet}
\left\{\begin{aligned}
&-\Delta w=-\Delta u\text{ in }\Omega\setminus\cup_{j=1}^N\overline{O_j},\\
&w=0\text{ on }\partial\Omega,\\
&w=0\text{ on }\partial O_j,\ j=1,\dots,N,\\
&[w]=0\text{ across }\partial O_j,\ j=1,\dots,N,\\
&-\Delta w=0\text{ in }\cup_{j=1}^NO_j.
\end{aligned}\right.
\end{equation}
We observe that the restriction of $w$, solution to \eqref{problem for w, Laplace-Dirichlet}, to the exterior of the objects satisfies a problem for a Poisson equation with homogeneous Dirichlet boundary conditions, and that its restriction to the interior of a given object solves a problem for the Laplace equation with a homogeneous Dirichlet boundary condition. The above problem thus makes sense for $u$ and $w$ in %\footnote{Note that the transmission conditions across the boundaries of the objects in the problem is indeed satisfied in $H$, due to the assumed regularity of these hypersurfaces and the fact that a function in $H^1(\Omega)$ is absolutely continuous on almost every rectifiable curve (see \cite{Hajlasz:2003}).}
 $H$, and is well-posed by virtue of the Lax--Milgram lemma. The well-posedness of the problem then implies that functions in the range of the projector vanish in the interiors of the objects, from which the characterisation of the subspace $M$ follows.

Using these facts and the weak formulation associated with the problem, one has
\begin{align*}
\forall u\in H,\ \forall w\in M,\ a(P_Mu,w)&=\int_{\Omega\setminus\cup_{j=1}^N\overline{O_j}}\grad P_Mu(x)\cdot\grad w(x)\,\mathrm{d}x\\
&=\left<-\Delta P_Mu,w\right>_{H^{-1}(\Omega\setminus\cup_{j=1}^N\overline{O_j}),H^1_0(\Omega\setminus\cup_{j=1}^N\overline{O_j})}\\
&=\left<-\Delta u,w\right>_{H^{-1}(\Omega\setminus\cup_{j=1}^N\overline{O_j}),H^1_0(\Omega\setminus\cup_{j=1}^N\overline{O_j})}\\
&=\int_{\Omega\setminus\cup_{j=1}^N\overline{O_j}}\grad u(x)\cdot\grad w(x)\,\mathrm{d}x\\
&=a(u,w),
\end{align*}
so that $P_M$ and $P_V=Id-P_M$ are orthogonal projectors.

Consider next the subset $V$ of $H$, orthogonal complement to $M$. One has that $P_V=Id_H-P_M$, so that, for any $u$ in $H$, the function $v=P_Vu$ is a weak solution to the following system
\[ %begin{equation}\label{problem for v, Laplace-Dirichlet}
\left\{\begin{aligned}
&-\Delta v=0\text{ in }\Omega\setminus\cup_{j=1}^N\overline{O_j},\\
&v=0\text{ on }\partial\Omega,\\
&v=u\text{ on }\partial O_j,\ j=1,\dots,N,\\
&[v]=0\text{ across }\partial O_j,\ j=1,\dots,N,\\
&-\Delta v=-\Delta u\text{ in }\cup_{j=1}^NO_j,
\end{aligned}\right.
\] %end{equation}
yielding the characterisation of $V$.\\
The orthogonality of the projectors $P_{M_i}$ and $P_{V_i}$, and the characterizations of the respective associated subspaces $M_i$ and $V_i$, with $i$ in $\{1,\dots,N\}$, are obtained in the same way.

Let us finally show that $V=\sum_{j=1}^NV_j$. Any element $v$ of $V$ may be represented by the sum of double layer potentials and Newton potentials, that is
\[
\text{ a. e. }x\in\Omega\setminus\cup_{j=1}^{N}\partial O_j,\ v(x)=\sum_{j=1}^N\int_{\partial O_j}G(x,y)\left[\frac{\partial v}{\partial\vec{n}}\right](y)\,\mathrm{d}S(y)-\sum_{j=1}^N\int_{O_j}G(x,y)\Delta v(y)\,\mathrm{d}y,
\]
where the kernel $G$ is the Green function\footnote{We recall that the Green function $G$ is such that
\[
\forall(x,y)\in\Omega\times\Omega,\ x\neq y,\ G(x,y)=\Phi(y-x)-\varphi_x(y),
\]
the function $\Phi$ being the fundamental solution of the Laplace equation %, given by
%\[
%\Phi(x)=-\frac{1}{2\pi}\,\log(\abs{x})\text{ for $d=2$ and }\Phi(x)=\frac{1}{d(d-2)V_d(1)}\frac{1}{\abs{x}^{d-2}}\text{ for $d\geq 3$},
%\]
%with $\abs{\cdot}$ the Euclidean norm over $\mathbb{R}^d$, $V_d(1)$ the volume of a ball of unit radius in $\mathbb{R}^d$, 
and $\varphi_x$ being a corrector function which, for a fixed $x$ in $\Omega$, satisfies $\Delta\varphi_x=0$ in $\Omega$ and $\varphi_x=\Phi(\cdot-x)$ on $\partial\Omega$. % The function $w$ is also said to be represented by a \emph{simple} (resp. \emph{double}) \emph{layer potential} when $\left[\frac{\partial w}{\partial\vec{n}}\right]=0$ (resp. $[w]=0$) across $\partial O_j$ for any integer $j$ in $\{1,\dots,N\}$.
} of the Laplace equation for the region $\Omega$, and the various integrals are understood in the sense of duality products. This provides the decomposition $v=\sum_{j=1}^Nv_j$ by setting
\[
\forall i\in\{1,\dots,N\},\ v_i(x)=\int_{\partial O_i}G(x,y)\left[\frac{\partial v}{\partial\vec{n}}\right](y)\,\mathrm{d}S(y)-\int_{O_j}G(x,y)\Delta v(y)\,\mathrm{d}y.
\]
For each integer $i$ in $\{1,\dots,N\}$, it is then easily seen that $v_i$ belongs to $V_i$, concluding the proof.
\end{proof}

\medskip

The assumptions of Theorems~\ref{convergence sequential} and~\ref{convergence averaged parallel} being satisfied in view of the above result, both the sequential and averaged  parallel form of the method converge to the solution.

\begin{remark}
In the present case, the orthogonality between the subspaces $V$ and $M$ (resp. $V_i$ and $M_i$) may be related to the orthogonality that exists between gradients of functions that are harmonic in a regular bounded open set $\Omega$ in $\mathbb{R}^d$ and gradients of functions in $H^1_0(\Omega)$, with respect to the $L^2(\Omega,\mathbb{R}^d)$-inner product.
\end{remark}

\smallskip

\paragraph{\textbf{Laplace problem with boundary conditions of the fourth type.}} 
By this name, we refer to the nonlocal boundary conditions used in boundary value problem \eqref{Laplace equation, 4th type problem}-\eqref{boundary conditions, 4th type problem}-\eqref{surface charge conditions, 4th type problem}-\eqref{outer Dirichlet condition, 4th type problem} considered in Subsection \ref{subsec:introMob}, which correspond to homogeneous conditions for the boundary operators
\[ %begin{equation}\label{eq:Bi_4th}
\mathcal{B}_{j}(u)=u_{|_{\partial O_j}}-\fint_{\partial O_j}u(x)\,\mathrm{d}S(x), \ j=1,\dots,N,
\] %end{equation}
where $\fint_{\partial O_j}u(x)\,\mathrm{d}S(x)$ denotes the mean of the function $u$ on the hypersurface $\partial O_j$.

As for the preceding Poisson--Dirichlet problem, we extend the boundary value problem to the interior of the objects by continuity, imposing that
\[
[u^*]=0\text{ across }\partial O_j,\ j=1,\dots,N.
\]
We may then look for a solution $u^*$ in $H=H^1_0(\Omega)$. The boundary conditions on the objects being homogeneous, a suitable lifting is any function in $H$ having constant value on the boundaries of the objects, for instance $\lifting\equiv0$. Finally, as the given conditions on the boundaries of the objects correspond to a source term with support on these boundaries (see \cite{Luke:1989,Hofer:2018}), we consequently require the initialisation $u^{(0)}$ to be a function in $H$ satisfying
\[
-\Delta u^{(0)}=\sum_{j=1}^NQ_j\frac{\mathcal{H}^{d-1}_{|_{\partial O_j}}}{\mathcal{H}^{d-1}(\partial O_j)}\text{ in }\Omega
\]
Note that such a function exists, as one can be constructed by summing solutions to one-object problems. 

\smallskip

We have the following result.

\begin{proposition} %\label{4th_case}
Let $H=H_0^1(\Omega)$, equipped with the inner product
\[
a(u,v)=\int_\Omega\grad u(x)\cdot\grad v(x)\,\mathrm{d}x.
\]
For the Laplace operator completed with fourth-type boundary conditions, the projectors $P_M$ and $P_V$ (resp. $P_{M_{i}}$ and $P_{V_{i}}$, with $i$ in $\{1,\dots,N\}$) introduced in Subsection~\ref{projection setting} are well-defined and orthogonal with respect to $a$, and their respective ranges are characterised as
\[
M=\left\{w\in H\,|\,\forall j\in\{1,\dots,N\},\ \exists c_j\in\mathbb{R},\ w=c_j\text{ in }\overline{O_j}\right\}
\]
and
\begin{equation}\label{subspace V, fourth type laplace}
V=\left\{v\in H\,|\,\Delta v=0\text{ in }\Omega\setminus\cup_{j=1}^N\overline{O_j},\ \forall j\in\{1,\dots,N\},\ \int_{\partial O_j}\frac{\partial v}{\partial n}(x)\,\mathrm{d}S(x)=0\right\}.
\end{equation}
\begin{multline*} %\label{MiVi4th}
\text{(resp. }M_i=\left\{w\in H\,|\,\exists c_i\in\mathbb{R},\ w=c_i\text{ in }\overline{O_i}\right\}\\\text{ and }V_i=\left\{v\in H\,|\,\Delta v=0\text{ in }\Omega\setminus\overline{O_i},\ \int_{\partial O_i}\frac{\partial v}{\partial n}(x)\,\mathrm{d}S(x)=0\right\},\ i=1,\dots,N\text{)}.
\end{multline*}
In addition, one has $V=\sum_{j=1}^N V_j$.
\end{proposition}
\begin{proof}
We proceed as in the proof of the preceding proposition. For any $u$ in $H$, the function $w=P_Mu$ is the weak solution to the following system
\[ %begin{equation}\label{exterior problem for w, 4th}
\left\{\begin{aligned}
&-\Delta w=-\Delta u\text{ in }\Omega\setminus\cup_{j=1}^N\overline{O_j},\\
&w=0\text{ on }\partial\Omega,\\
&\int_{\partial{O_j}}\frac{\partial w}{\partial n}(x)\,\mathrm{d}S(x)=\int_{\partial{O_j}}\frac{\partial u}{\partial n}(x)\,\mathrm{d}S(x), \ j=1,\dots,N,\\
&w_{|\partial O_j}-\fint_{\partial O_j}w(x)\,\mathrm{d}S(x)=0,\text{ on }\partial O_j,\ j=1,\dots,N,\\
&[w]=0\text{ across }\partial O_j,\ j=1,\dots,N,\\
&-\Delta w=0\text{ in }\cup_{j=1}^NO_j.
\end{aligned}\right.
\] %end{equation}
Existence and uniqueness of a solution to the exterior problem are a well-known fact (see~\cite{Li:1989} for instance). The well-posedness of the interior Laplace--Dirichlet problem satisfied by the solution then implies that it is constant in each of the objects, a fact from which the characterisation of the space $M$ follows.

Let us now prove that that $P_M$ is an orthogonal projector. One has, integrating by parts and using the problem defining the projector,
\begin{align*}
\forall u\in H,\ \forall w\in M,\ a(P_Mu,w)&=\int_{\Omega\setminus\cup_{j=1}^N\overline{O_j}} \grad P_Mu(x)\cdot\grad w(x)\,\mathrm{d}x\\
&=\left<-\Delta P_Mu,w\right>_{H^{-1}(\Omega\setminus\cup_{j=1}^N\overline{O_j}),H^1_0(\Omega\setminus\cup_{j=1}^N\overline{O_j})}+\sum_{j=1}^N\left<\frac{\partial P_Mu}{\partial n},w\right>_{H^{-1/2}(\partial{O_j}),H^{1/2}(\partial{O_j})}\\
&=\left<-\Delta u,w\right>_{H^{-1}(\Omega\setminus\cup_{j=1}^N\overline{O_j}),H^1_0(\Omega\setminus\cup_{j=1}^N\overline{O_j})}+\sum_{j=1}^Nw_{|_{\overline{O_j}}}\int_{\partial{O_j}}\frac{\partial u}{\partial n}(x)\,\mathrm{d}S(x)\\
&=\int_{\Omega\setminus\cup_{j=1}^N\overline{O_j}}\grad u(x)\cdot\grad w(x)\,\mathrm{d}x\\
&=a(u,w),
\end{align*}
since the $w$ is constant on the boundaries of the objects.

To characterise the subspace $V$, let $v$ be an element of the set on the right hand side of the identity \eqref{subspace V, fourth type laplace}. For any function $w$ in $M$, one has
\[
a(v,w)=\left<-\Delta v,w\right>_{H^{-1}(\Omega\setminus\cup_{j=1}^N\overline{O_j}),H^1_0(\Omega\setminus\cup_{j=1}^N\overline{O_j})}+\sum_{j=1}^Nw_{|_{\overline{O_j}}}\int_{\partial{O_j}}\frac{\partial v}{\partial n}(x)\,\mathrm{d}S(x)=0
\]
Conversely, let $v$ be an element of the orthogonal complement of $M$. Then, for any function $w$ in $M$, one has
\[
0=a(v,w)=\left<-\Delta v,w\right>_{H^{-1}(\Omega\setminus\cup_{j=1}^N\overline{O_j}),H^1_0(\Omega\setminus\cup_{j=1}^N\overline{O_j})}+\sum_{j=1}^Nw_{|_{\overline{O_j}}}\int_{\partial{O_j}}\frac{\partial v}{\partial n}(x)\,\mathrm{d}(x)S.
\]
Choosing a function $w$ in $M$ which vanishes in the objects, one obtains that $-\Delta v=0$ in $\Omega\setminus\cup_{j=1}^N\overline{O_j}$ in a weak sense. One then concludes by using a function $w$ which vanishes in all the objects except the $j$th one and by varying the integer $j$ to obtain the remaining conditions.\\
Here again, the orthogonality of the projectors $P_{M_i}$, and the characterizations of the subspaces $M_i$ and their orthogonal complements $V_i$, with $i$ in $\{1,\dots,N\}$, are obtained similarly.

It remains to prove $V=\sum_{j=1}^NV_j$. One has
\[
\forall v\in V,\text{ a. e. }x\in\Omega\setminus\cup_{j=1}^{N}\partial O_j,\ v(x)=\sum_{j=1}^N\int_{\partial O_j}G(x,y)\left[\frac{\partial v}{\partial\vec{n}}\right](y)\,\mathrm{d}S(y)-\sum_{j=1}^N\int_{O_j}G(x,y)\Delta v(y)\,\mathrm{d}y,
\]
with
\[
\forall j\in\{1,\dots,N\},\ \int_{\partial O_j}\frac{\partial v}{\partial\vec{n}}(x)\,\mathrm{d}S(x)=0.
\]
Setting
\[
\forall i\in\{1,\dots,N\},\ v_i(x)=\int_{\partial O_i}G(x,y)\left[\frac{\partial v}{\partial\vec{n}}\right](y)\,\mathrm{d}S(y)-\int_{O_i}G(x,y)\Delta v(y)\,\mathrm{d}y
\]
then provides the adequate decomposition.
\end{proof}

\medskip

It follows from this result that Theorems~\ref{convergence sequential} and~\ref{convergence averaged parallel} apply, so that the sequential and averaged parallel versions converge.

\begin{remark}
Continuing with the electrostatics analogy started in Subsection \ref{subsec:introMob}, a function $v$ in the subspace $V_i$ is usually said to be the potential an \emph{electric dipole}, since the quantity
\[
\int_{\partial{O_i}}\frac{\partial v}{\partial n}(x)\,\mathrm{d}S(x),
\]
which corresponds to the total charge of the $i$th object, vanishes.
\end{remark}

%\paragraph{\textbf{Neumann boundary conditions.}}
%\input{practicle_cases_ju3}

\smallskip

\subsubsection{The mobility problem for the Stokes equations}\label{appl:Luke}
In \cite{Luke:1989}, Luke analysed the sequential form of the method of reflections and proved its convergence when applied to the solution of a system of equations modelling the motion of a sedimenting suspension in a container, the so-called mobility problem for the Stokes equations. It is similar to, but more complex than, the Laplace problem with boundary condition of the fourth type we previously dealt with. We show how this problem fits into the orthogonal projection framework previously introduced. Proofs of the statement are omitted as they can be found in \cite{Luke:1989} or adapted from \cite{Hoefer:2018,Niethammer}. Note that a similar analysis for the Stokes equations with Dirichlet boundary conditions and a particle configuration set in the whole space can be found in \cite{Hofer:2018}.

Denoting by $\Omega$ the container, that is, a bounded, connected, open set of $\mathbb{R}^3$ with a smooth boundary, and by $O_i\subset\Omega$, with $i$ in $\{1,\dots,N\}$, the rigid particles of arbitrary shape, which are connected open sets with smooth boundaries, such that their closures are non overlapping, the set $\cup_{j=1}^N O_j$ is called the solid phase of the suspension, while the set $\Omega\backslash\overline{\cup_{j=1}^N O_j}$ is called the fluid phase.

The problem is extended flow inside the particles by requiring the flow field to be continuous across the particle boundaries and the Stokes equations to also hold inside the particles, the inertialess motion of the rigid particles due to externally imposed forces and torques is then described by the fluid velocity field $\vec{u}$ and the pressure field $p$ satisfying the interior-exterior transmission problem
\begin{align*}
&-\mu\,\Delta u+\grad p=0\text{ in }\Omega\setminus\cup_{j=1}^N\overline{O_j},\\
&\div u=0\text{ in }\Omega\setminus\cup_{j=1}^N\overline{O_j},\\
&u=0\text{ on }\partial\Omega,\\
&\forall x\in\partial O_j,\ u(x)=U_{O_j}+\omega_{O_j}\times(x-x_{O_j}),\ j=1,\dots,N,\\
&\int_{\partial O_j}[\vec{\sigma}(u)]\cdot\vec{n}(y)\,\mathrm{d}S(y)=F_{O_j},\ j=1,\dots,N,\\
&\int_{\partial O_j}(y-x_{O_j})\times[\vec{\sigma}(u)]\cdot\vec{n}(y)\,\mathrm{d}S(y)=\tau_{O_j},\ j=1,\dots,N,\\
&[u]=0\text{ across }\partial O_j,\ j=1,\dots,N,\\
&-\mu\,\Delta u+\grad p=0\text{ in }\cup_{j=1}^N O_j,\\
&\div u=0\text{ in }\cup_{j=1}^NO_j,
\end{align*}
where $\mu$ is the kinematic viscosity coefficient (which will be set to $1$ in what follows for simplicity), $\vec{\sigma}(u)$ is the stress tensor, $\vec{\sigma}(u)=\mu\,\left(\grad u+(\grad u)^T\right)-p\,\vec{\mathrm{I}}$, $\vec{n}$ is the outward pointing unit normal vector to the particle boundaries, and the instantaneous linear and angular velocities $U_{O_i}$ and $\omega_{O_i}$ of the $i$th particle are unknowns in $\mathbb{R}^3$ to be determined (along with the fluid velocity and pressure), whereas the centres of mass $x_{O_i}$ of the particles, the total forces $F_{O_i}$ and the total torques $\tau_{O_i}$, with $i$ in $\{1,\dots,N\}$, applied to the particles are given\footnote{For instance, for suspensions sedimenting in a uniform gravitational field, one has $F_{O_i}=m_{O_i}\vec{G}$, where $m_{O_i}$ is the mass of the particle adjusted for buoyancy and $\vec{G}$ is the gravitational acceleration, and $\tau_{O_i}=0$.}. In the above system, the trivial flow is used as an admissible lifting of the homogeneous boundary data.

The variational formulation of the above problem allows to reduce it to that of finding solely the velocity field $\vec{u}$ in the space
\[
H=\vec{H}^1_{0,\sigma}(\Omega):=\{v\in H^1_0(\Omega,\mathbb{R}^3)\,|\,\div v=0\text{ in }\mathscr{D}'(\Omega)\},
\]
as the pressure $p$ can recovered (up to a constant) from it, which is equipped with the bilinear, symmetric, continuous and coercive form
\[
\forall(u,v)\in H\times H,\ a(u,v)=\int_{\Omega}\grad u(x):\grad v(x)\,\mathrm{d}x.
\]

As an initialisation, one chooses a field $u^{(0)}$ satisfying the equations of the above system except for the boundary conditions on the surface of the particles, meaning that associated the fluid flow will satisfy the constraints of the forces acting on the particles but will fail to have a rigid motion in the particles.

In such a setting, the elements of the subspace $M$ are the elements of $H$ which achieve a rigid motion in the particles, \textit{i.e.},
\[
M=\left\{w\in H\,|\,\exists U\in(\mathbb{R}^3)^N,\ \exists\omega\in(\mathbb{R}^3)^N,\ \forall j\in\{1,\dots,N\},\ \forall x\in\overline{O_j},\ w(x)=U_j+\omega_j\times(x-x_{O_j})\right\},
\]
the orthogonal complement of $M$ being
\begin{multline*}
V=\left\{v\in H\ |\ \exists p\in L^2(\Omega),\ \Delta v+\grad p=0\text{ in }\Omega\setminus\cup_{j=1}^N\overline{O_j},\right.\\\left.\forall j\in\{1,\dots,N\},\ \int_{\partial O_j}[\vec{\sigma}(v)]\cdot\vec{n}(y)\,\mathrm{d}S(y)=0\text{ and }\int_{\partial O_j}(y-x_{O_j})\times[\vec{\sigma}(v)]\cdot\vec{n}(y)\,\mathrm{d}S(y)=0\right\}.
\end{multline*}
It then follows that, for any integer $i$ in $\{1,\dots,N\}$,
\[
\forall i\in\{1,\dots,N\},\ M_i=\left\{w\in H\,|\,\exists U\in\mathbb{R}^3,\ \exists\omega\in\mathbb{R}^3,\ \forall x\in\overline{O_i},\ w(x)=U_i+\omega_i\times(x-x_{O_i})\right\},
\]
and
\begin{multline*}
\forall i\in\{1,\dots,N\},\ V_i=\left\{v\in H\ |\ \exists p\in L^2(\Omega),\ \Delta v+\grad p=0\text{ in }\Omega\setminus\overline{O_i},\right.\\\left.\int_{\partial O_i}[\vec{\sigma}(v)]\cdot\vec{n}(y)\,\mathrm{d}S(y)=0\text{ and }\int_{\partial O_i}(y-x_{O_i})\times[\vec{\sigma}(v)]\cdot\vec{n}(y)\,\mathrm{d}S(y)=0\right\},
\end{multline*}
the elements of the subspaces $V_i$ being the hydrodynamic analogues of the electric dipoles. The fact that the sum $\sum_{j=1}^N V_j$ is closed is established by Luke.

\begin{remark}
When $\Omega=\mathbb{R}^3$ and the particles are identical spheres of radius $R$, one can characterise in a more precise way both the projections onto $M_i$ and the functions in $V_i$, with $i$ in $\{1,\dots,N\}$, using closed-form expressions, like Stokes' law, to explicitly compute the drag force exerted on a particle (see \cite{Hoefer:2018,Niethammer}). In such a case, $H$ is the subspace $\vec{H}^1_{\sigma}$ of all divergence free functions in the homogeneous Sobolev space $\dot{H}^1(\mathbb{R}^3,\mathbb{R}^3)$ (defined as the closure of $\mathscr{C}^\infty_c(\mathbb{R}^3,\mathbb{R}^3)$ with respect to the $L^2$-norm of the gradient) and, for a given function $u$ in $H$, the projection of $u$ onto $M_i$ satisfies $P_{M_i}u(x)=U+\omega_{O_i}\times(x-x_{O_i})$ for any $x$ in $\overline{O_i}$, with
\[
U=\frac{1}{4\pi R^2}\int_{\partial O_i}u(x)\,\mathrm{d}S(x)\text{ and }\omega_{O_i}=\frac{3}{8\pi R^4}\int_{\partial O_i}(x-x_{O_i})\times u(x)\,\mathrm{d}S(x).
\]
Moreover, for any function $v$ in $V_i$, one has
\[
\int_{\partial O_i}v(x)\,\mathrm{d}S(x)=0\text{ and }\int_{\partial O_i}(x-x_{O_i})\times v(x)\,\mathrm{d}S(x)=0.
\]
\end{remark}

\subsection{An example of non-orthogonal case} %\label{Sec:nonortho}
We now consider a problem for the Laplace equation set in an unbounded domain, \textcolor{black}{each} object having either a Dirichlet or a Neumann condition imposed on its boundary. For such a configuration, we were unable to prove the orthogonality of the projection operators, and cases of divergence for the sequential form of the method are actually observed in numerical tests in dimension two (see Subsection \ref{divergence example}). This leads us to conjecture that the convergence theory proposed in the present work does not apply in this case, which is thus called ``non-orthogonal''.

Nevertheless, assuming as in Proposition~\ref{prop:banach} that $H$ is a Banach space, one can find sufficient conditions for the convergence of the sequential version of the method. To do this, we follow the approach used by Balabane in \cite{Balabane:2004} for the parallel form of the method of reflections, and prove that, under a certain geometrical condition, the series defining the approximation converges to the solution to the boundary value problem under consideration.

In the present subsection, this domain $\Omega$ is equal to $\mathbb{R}^3$ and the $N$ objects are disjoints compact sets in $\Omega$, with boundaries of class $\mathscr{C}^2$. Given two positive integers $N_D$ and $N_N$ such that $N=N_D+N_N$, the objects associated with a Dirichlet boundary condition are numbered from $1$ to $N_D$ and that the ones associated with a Neumann boundary condition are numbered form $N_D+1$ to $N$. Then, given $N_D$ functions $b^D_i$, $i=1,\dots,N_D$, and $N_N$ functions $b^N_i$, $i=1,\dots,N_N$, we look for a function $u$ satisfying
\begin{equation}\label{mixed problem in R3}
\begin{array}{l}
-\Delta u=0\text{ in }\mathbb{R}^3\setminus\cup_{j=1}^N\overline{O_j},\\
u=b^D_j\text{ on }\partial O_j,\ j=1,\dots,N_D,\\
\dfrac{\partial u}{\partial\vec{n}}=b^N_j\text{ on }\partial O_{N_D+j},\ j=1,\dots,N_N.
\end{array}
\end{equation}
Observe that \textcolor{black}{this} boundary value problem is an exterior one (it is set in the complement of a union of bounded sets), and necessitates the introduction of a weighted Sobolev space to properly define its solution (see \cite{Amrouche:1997} or \cite[XI, B]{Dautray:2000}). For any subset $\omega$ of $\Omega$, we set
\[
W(\omega)=\left\{v\,:\,\omega\to\mathbb{R}\,|\,\int_{\Omega}\frac{v^2(x)}{1+\norm{x}^2}\,\mathrm{d}x<+\infty,\ \grad v\in L^2(\omega,\mathbb{R}^3)\right\}.
\]

A first step in proving the convergence of the method is to establish some boundary estimates for the relections.

\begin{lemma}\label{lemma: estimate}
Consider the sequence of reflections $(u^{(k)}_i)_{k\in\mathbb{N},i\in\{1,\dots,N\}}$ generated by the sequential form of the method of reflections applied to the solution of problem \eqref{mixed problem in R3} and define the associated sequence of scalars
\[
\forall k\in\mathbb{N}^*,\ \forall(i,j)\in\{1,\dots,N\}^2,\ a^{(k)}_{i,j}=\dnorm{u^{(k)}_i}^2_{H^1(\partial O_j)}+\dnorm{\frac{\partial u^{(k)}_i}{\partial\vec{n}}}^2_{L^2(\partial O_j)}.
\]
Then, for any pair of distinct integers $i$ and $j$ in $\{1,\dots,N\}$, there exists a positive constant $\kappa_{i,j}$ depending on the geometry of the problem such that
\begin{equation}\label{ineq:i_est}
\forall k\in\mathbb{N}^*,\ a^{(k+1)}_{i,j}\leq N\kappa_{i,j}\left(\sum_{m=1}^{i-1} a^{(k+1)}_{m,i}+\sum_{m=i+1}^{N}a^{(k)}_{m,i}\right).
\end{equation}
\end{lemma}
\begin{proof}
First, we remark that, for any integer $k$ in $\mathbb{N}^*$, the datum of the boundary value problem defining a reflection is the trace (for a reflection associated to an object with a Dirichlet boundary condition) or the trace of the normal derivative (for a reflection associated to an object with a Neumann condition) on an interior curve of a sum of harmonic functions. It follows from Weyl's lemma on the interior regularity of harmonic functions and from results in Chapter 2 of \cite{Grisvard:1985} that these reflections enjoy smoothness properties which, using a trace continuity theorem (see Chapter 1 of \cite{Grisvard:1985} for instance), allow to show there exist positive constants $C_i$ depending only on the geometry such that
\begin{equation}\label{ineq:diri}
\forall k\in\mathbb{N}^*,\ \forall i\in\{1,\dots,N_D\},\ \dnorm{\frac{\partial u^{(k)}_i}{\partial\vec{n}}}_{L^2(\partial O_i)}\leq C_i\,\dnorm{u^{(k)}_i}_{H^1(\partial O_i)},
\end{equation}
and
\[ %begin{equation}\label{ineq:neu}
\forall k\in\mathbb{N}^*,\ \forall i\in\{N_D+1,\dots,N\},\ \dnorm{u^{(k)}_i}_{H^1(\partial O_i)}\leq C_i\,\dnorm{\frac{\partial u^{(k)}_i}{\partial\vec{n}}}_{L^2(\partial O_i)}.
\] %end{equation}
From now on, we shall assume, without loss of generality, that $C_i\geq 1$, with $i$ in $\{1,\dots,N\}$.

Next, for any integers $k$ in $\mathbb{N}^*$and $i$ in $\{1,\dots,N\}$, the reflection $u^{(k)}_i$ is extended to the whole of $\Omega$ by requiring its extension $\tilde{u}^{(k)}_i$ to belong to the weighted Sobolev space $W(\Omega)$, to satisfy the same equations as $u^{(k)}_i$ in the complement of $O_i$ and the Laplace equation in $O_i$, and to have a vanishing jump across $\partial O_i$ if $1\leq i\leq N_D$, or a vanishing jump of its normal derivative across $\partial O_i$ if $N_D+1\leq i\leq N$. Note that the interior problem associated with a reflection satisfying a Neumann boundary condition is indeed well-posed, since its normal derivative has a zero mean value on $\partial O_i$. 
%As a consequence,  (resp. $[\partial_n u_i]$) does not necessarily cancels for Neumann (resp. Dirichlet) type objects. In such a case and for $y\in \partial O_i$, we denote by $u_{i}(y)$ (resp. $\partial_n u_{i}(y)$) the exterior trace value of $u_i$ (resp. $\partial_n u_i$) at the point $y$. In the same way, $u_{i|\partial O_i}$ (resp. $\partial_n u_{i|\partial O_i}$) refers to the interior traces of $u_i$ (resp. $\partial_n u_i$) on the boundary $\partial O_i$.
We may then define the space
\begin{multline*}
H=H=\Big\{v\in W(\Omega)\,|\,
-\Delta v=0\text{ in }\mathbb{R}^3\backslash\cup_{j=1}^N{\partial O_j},\ \left[v\right]=0\text{ on }\partial O_j,\ j=1,\dots,N_D,\\
\left[\frac{\partial v}{\partial\vec{n}	}\right]=0\text{ on }\partial O_j,\ \int_{\partial O_j}\frac{\partial v}{\partial\vec{n}}(y)\,\mathrm{d}S(y)=0\text{ and }\int_{\partial O_j}v_{|_{O_j}}(y)\,\mathrm{d}S(y)=0,\ j=N_D+1,\dots,N\Big\},
\end{multline*}
and its subspaces
\[
\forall i\in\{1,\dots,N\},\ V_i=\left\{w\in H\,|\,-\Delta w=0\text{ in }\Omega\setminus\partial O_i\right\}.
\]
Functions $v$ in $H$ have an explicit formulation in terms of their jumps on the boundaries $\partial O_i$, with $i$ in $\{1,\dots,N\}$, due to the following integral representation
\begin{equation}\label{eq:repre}
\forall x\in\Omega\setminus\cup_{j=1}^{N}\partial O_j,\ v(x)=\frac{1}{4\pi}\sum_{j=N_D+1}^N\int_{\partial O_j}\frac{(x-y)\cdot\vec{n}(y)}{\norm{x-y}^3}[v(y)]\,\mathrm{d}S(y)-\frac{1}{4\pi}\sum_{j=1}^{N_D}\int_{\partial O_j}\frac{1}{\norm{x-y}}\left[\frac{\partial v}{\partial\vec{n}}(y)\right]\,\mathrm{d}S(y),
\end{equation}
and the same goes for their gradient,
\begin{multline}\label{eq:dnrepre}
\forall x\in\Omega\setminus\cup_{j=1}^{N}\partial O_j,\ \grad v(x)=\frac{1}{4\pi}\sum_{j=N_D+1}^N\int_{\partial O_j}\frac{1}{\norm{x-y}^3}\left(\vec{n}(y)-3\frac{(x-y)\cdot\vec{n}(y)}{\norm{x-y}^2}(x-y)\right)[v(y)]\,\mathrm{d}S(y)\\+\frac{1}{4\pi}\sum_{j=1}^{N_D}\int_{\partial O_j}\frac{x-y}{\norm{x-y}^3}\left[\frac{\partial v}{\partial\vec{n}}(y)\right]\,\mathrm{d}S(y).
\end{multline}

\smallskip

For any integer $i$ in $\{1,\dots,N\}$, let us denote by $S_i$ the area of $\partial O_i$, that is $S_i=\mathcal{H}^2(\partial(O_i)$, and set, for any pair of distinct integers $i$ and $j$ in $\{1,\dots,N\}$, $d_{i,j}=\min_{(x,y)\in\partial O_i\times\partial O_j}\,\norm{x-y}$. We will deal differently with the reflection according to the type of boundary condition it satisfies on the associated object.

For a reflection associated with an object with an imposed Dirichlet boundary condition, that is for an integer $i$ in $\{1,\dots,N_D\}$, the jump condition across $\partial O_i$ yields
\[
\forall y\in\partial O_i,\ \left[\frac{\partial\tilde{u}^{(k+1)}_i}{\partial\vec{n}}(y)\right]=\frac{\partial u^{(k+1)}_i}{\partial\vec{n}}(y)+\sum_{m=1}^{i-1}\frac{\partial u^{(k+1)}_m}{\partial\vec{n}}(y)+\sum_{m=i+1}^{N}\frac{\partial u^{(k)}_m}{\partial\vec{n}}(y).
\]
On the other hand, since the function $\tilde{u}^{(k)}_i$ belongs to $V_i$, it follows from the integral representation formula \eqref{eq:repre} that
\[
\forall x\in\mathbb{R}^3\setminus\partial O_i,\ \tilde{u}^{(k+1)}_i(x)=-\frac{1}{4\pi}\int_{\partial O_i}\frac{1}{\norm{x-y}}\left[\frac{\partial\tilde{u}^{(k+1)}_i}{\partial\vec{n}}(y)\right]_{\partial O_i}\,\mathrm{d}S(y),
\]
hence
\begin{multline*}
\forall x\in\mathbb{R}^3\setminus\partial O_i,\ \tilde{u}^{(k+1)}_i(x)=-\frac{1}{4\pi}\int_{\partial O_i} \frac{1}{\norm{x-y}}\frac{\partial u^{(k+1)}_i}{\partial\vec{n}}(y)\,\mathrm{d}S(y)-\frac{1}{4\pi}\sum_{m=1}^{i-1}\int_{\partial O_i}\frac{1}{\norm{x-y}}\frac{\partial u^{(k+1)}_m}{\partial\vec{n}}(y)\,\mathrm{d}S(y)\\
-\frac{1}{4\pi}\sum_{m=i+1}^{N}\int_{\partial O_i}\frac{1}{\norm{x-y}}\frac{\partial u^{(k)}_m}{\partial\vec{n}}(y)\,\mathrm{d}S(y).
\end{multline*}
By means of the Cauchy--Schwarz inequality, one has, for any integer $j$ in $\{1,\dots,N\}$ distinct from $i$,
\[
\forall x\in\partial O_j,\ \abs{u^{(k+1)}_i(x)}\leq\frac{\sqrt{S_i}}{4\pi d_{i,j}}\left(\dnorm{\frac{\partial u^{(k+1)}_i}{\partial\vec{n}}}_{L^2(\partial O_i)}+\sum_{m=1}^{i-1}\dnorm{\frac{\partial u^{(k+1)}_m}{\partial\vec{n}}}_{L^2(\partial O_i)}+\sum_{m=i+1}^{N} \dnorm{\frac{\partial u^{(k)}_m}{\partial\vec{n}}}_{L^2(\partial O_i)}\right),
\]
so that using \eqref{ineq:diri}, squaring both sides of the inequality and integrating over $\partial O_j$, we obtain
\begin{equation}\label{eqD1}
\dnorm{u^{(k+1)}_i}^2_{L^2(\partial O_j)}\leqslant\frac {N S_iS_jC_i^2}{(4\pi)^2 d_{i,j}^2}\left(\sum_{m=1}^{i-1}a^{(k+1)}_{m,i}+\sum_{m=i+1}^{N}a^{(k)}_{m,i}\right).
\end{equation}

In addition, repeating these computations starting from~\eqref{eq:dnrepre}, we get similar estimates the normal and tangential derivative traces,
\begin{align}\label{eqD2}
\dnorm{\frac{\partial u^{(k+1)}_i}{\partial\vec{n}}}^2_{L^2(\partial O_j)}&\leq\frac {NS_iS_jC_i^2}{(4\pi)^2 d_{i,j}^4}\left(\sum_{m=1}^{i-1} a^{(k+1)}_{m,i}+\sum_{m=i+1}^{N}a^{(k)}_{m,i}\right),\\
\dnorm{\frac{\partial u^{(k+1)}_i}{\partial\vec{\tau}}}^2_{L^2(\partial O_j)}&\leq\frac{NS_iS_jC_i^2}{(4\pi)^2d_{i,j}^4}\left(\sum_{m=1}^{i-1} a^{(k+1)}_{m,i}+\sum_{m=i+1}^{N}a^{(k)}_{m,i}\right).\label{eqD2bis}
\end{align}

Likewise, for a reflection $\tilde{u}^{(k)}_i$ satisfying a Neumann boundary condition on $\partial O_i$, that is for any integer $i$ in $\{N_D+1,\dots,N\}$, one has, one has
\[
\forall y\in\partial O_i,\ \left[\tilde{u}^{(k+1)}_i(y)\right]=u^{(k+1)}_i(y)+\sum_{m=1}^{i-1} u^{(k+1)}_m(y)+\sum_{m=i+1}^{N}u^{(k)}_m(y)+\lambda_i^{(k+1)},
\]
where $\lambda_i^{(k+1)}$ is a real number chosen in such a way that the condition \[\int_{\partial O_i}\tilde{u}^{(k+1)}_i{}_{|_{O_i}}(y)\,\mathrm{d}S(y)=0\] is satisfied. The integral representation formula then gives
\begin{multline*}
\forall x\in\mathbb{R}^3\setminus\partial O_i,\ \tilde{u}^{(k+1)}_i(x)=\frac{1}{4\pi}\int_{\partial O_i}\frac{(x-y)\cdot\vec{n}(y)}{\norm{x-y}^3} u^{(k+1)}_{i}(y)\,\mathrm{d}S(y)+\frac{1}{4\pi}\sum_{m=1}^{i-1}\int_{\partial O_i} \frac{(x-y)\cdot\vec{n}(y)}{\norm{x-y}^3}u^{(k+1)}_m(y)\,\mathrm{d}S(y)\\
+\frac{1}{4\pi}\sum_{m=i+1}^{N}\int_{\partial O_i}\frac{(x-y)\cdot\vec{n}(y)}{\norm{x-y}^3}u^{(k)}_m(y)\,\mathrm{d}S(y)+\frac{\lambda_i^{(k+1)}}{4\pi}\int_{\partial O_i}\frac{(x-y)\cdot\vec{n}(y)}{\norm{x-y}^3}\,\mathrm{d}S(y).
\end{multline*}
For any integer $j$ in $\{1,\dots,N\}$ distinct from $i$, one then has
\[
\forall x\in\partial O_j,\ \abs{u^{(k+1)}_i(x)}\leq\frac{\sqrt{S_i}}{4\pi d_{i,j}^2}\left(\dnorm{u^{(k+1)}_i}_{L^2(\partial O_i)}+\sum_{m=1}^{i-1}\dnorm{u^{(k+1)}_m}_{L^2(\partial O_i)}+\sum_{m=i+1}^{N}\dnorm{u^{(k)}_m}_{L^2(\partial O_i)}\right),
\]
from which one gets
\begin{equation}\label{eqD3}
\dnorm{u^{(k+1)}_i}^2_{L^2(\partial O_j)}\leq\frac{NS_iS_jC_i^2}{(4\pi)^2 d_{i,j}^4}\left(\sum_{m=1}^{i-1} a^{(k+1)}_{m,i}+\sum_{m=i+1}^{N}a^{(k)}_{m,i}\right).
\end{equation}
In the same manner, one may obtain the following estimates
\begin{align}\label{eqD4}
\dnorm{\frac{\partial u^{(k+1)}_i}{\partial\vec{n}}}^2_{L^2(\partial O_j)}&\leq\frac{NS_iS_jC_i^2}{\pi^2 d_{i,j}^6}\left(\sum_{m=1}^{i-1} a^{(k+1)}_{m,i}+\sum_{m=i+1}^{N}a^{(k)}_{m,i}\right),\\
\dnorm{\frac{\partial u^{(k+1)}_i}{\partial\vec{\tau}}}^2_{L^2(\partial O_j)}&\leq\frac{NS_iS_jC_i^2}{\pi^2d_{i,j}^6}\left(\sum_{m=1}^{i-1}a^{(k+1)}_{m,i}+\sum_{m=i+1}^Na^{(k)}_{m,i}\right).\label{eqD4bis}
\end{align}
Finally, summing estimates \eqref{eqD1}, \eqref{eqD2} and \eqref{eqD2bis} on the one hand, and estimates \eqref{eqD3}, \eqref{eqD4} and \eqref{eqD4bis} on the other hand, setting
\begin{equation}\label{def:kappa}
\kappa_{i,j}=\frac{S_iS_j{C_i}^2}{(4\pi)^2}\min\left\{\frac{1}{d_{i,j}^2}+\frac{2}{d_{i,j}^4},\frac{1}{d_{i,j}^4}+\frac{16}{d_{i,j}^6}\right\},
\end{equation}
and using the fact that the extension $\tilde{u}^{(k)}_i$ coincides with $u^{(k)}_i$ outside of $O_i$, we easily see that the claim holds whatever the type of condition imposed on the object boundary.
\end{proof}

\smallskip

A convergence criterion can now be stated.

\begin{theorem}
Let $\kappa(N)=\max\{\kappa_{i,j}\,|\,(i,j)\in\{1,\dots,N\}^2,\ i\neq j\}$, $\kappa_{i,j}$ being defined by \eqref{def:kappa}, and assume that
\begin{equation}\label{ineq:cvcond}
N(N-1)\kappa(N)<1.
\end{equation}
Then, \textcolor{black}{for any integer $i$ in $\{1,\dots,N\}$, the series $\sum_{k=0}^{+\infty}u_i^{(k)}$, generated by the sequential form of the method of reflections applied to the solution of problem \eqref{mixed problem in R3}, converges in $W(\Omega\setminus\overline{O_i})$, and its limits $u_i$ is such that the restriction of the sum $\sum_{i=1}^Nu_i$ to $\Omega\setminus\cup_{j=1}^N\overline{O_j}$ is the unique solution to problem \eqref{mixed problem in R3} in $W(\Omega\setminus\cup_{j=1}^N\overline{O_j})$.}
\end{theorem}
\begin{proof}
Define the sequence $(\alpha_\ell)_{\ell\in\mathbb{N}^*}$ by
\[
\forall k\in\mathbb{N}^*,\ \forall i\in\{1,\dots,N\},\ \alpha_{i+(k-1)N}=\max_{j\in\{1,\dots,N\}\setminus\{i\}}\,a^{(k)}_{i,j},
\]
the coefficients $a_{i,j}^{(k)}$ being defined in Lemma \ref{lemma: estimate}. Setting $\ell=i+kN$ for any integers $k$ in $\mathbb{N}^*$ and $i$ in $\{1,\dots,N\}$, one has, owing to \eqref{ineq:i_est}, 
\[
\alpha_\ell\leq N(N-1)\kappa(N)\underset{m\in\{1,\dots,N-1\}}{\max}\,\alpha_{\ell-m}\leq N(N-1)\kappa(N)\underset{m\in\{0,\dots,N-1\}}{\max}\,\alpha_{\ell-m}.
\]
For $\ell>2N$, denoting $m_0=\underset{m\in\{0,\dots,N-1\}}{\arg\max}\,\alpha_{\ell-m}$, one then gets
\[
\alpha_{\ell-m_0}\leq N(N-1)\kappa(N)\underset{m\in\{0,\dots,N-1\}}{\max}\,\alpha_{\ell-m_0-m},
\]
and, due to condition \eqref{ineq:cvcond},
\[
N(N-1)\kappa(N)\underset{m\in\{0,\dots,N-1-m_0\}}{\max}\,\alpha_{\ell-m_0-m}<\underset{m\in\{m_0,\dots,N-1\}}{\max}\,\alpha_{\ell-m}\leq\alpha_{\ell-m_0}.
\]
This implies that
\[
\alpha_{\ell-m_0}\leq N(N-1)\kappa(N)\max_{m=N-m_0,\dots,N-1}\,\alpha_{\ell-m_0-m}\leq N(N-1)\kappa(N)\max_{m=N,\cdots,N-1+m_0}\,\alpha_{\ell-m},
\]
so that finally
\[
\underset{m\in\{0,\dots,N-1\}}{\max}\,\alpha_{\ell-m}\leq N(N-1)\kappa(N)\max_{m=0,\cdots,N-1}\,\alpha_{\ell-N-m}.
\]
As a consequence, for any integer $i$ in $\{1,\dots,N\}$, the series $\sum_{k\in\mathbb{N}^*} u_i^{(k)}$ is convergent on the boundary of $\Omega\setminus\overline{O_i}$, and thus in $W(\Omega\setminus\cup_{j=1}^N\overline{O_j})$. Moreover, one can check that its limit $u_i$ satisfies
\[
\left\{\begin{array}{l}
-\Delta u_i=0\text{ in }\Omega\setminus\cup_{j=1}^N\overline{O_j},\\
u_i=U_i-\sum_{j=1}^{i-1}u_j-\sum_{j=i+1}^Nu_j\text{ on }\partial O_i,
\end{array}\right.
\]
if the integer $i$ belongs to $\{1,\dots,N_D\}$, or 
\[
\left\{\begin{array}{l}
-\Delta u_i=0\text{ in }\Omega\setminus\cup_{j=1}^N\overline{O_j},\\
\dfrac{\partial u_i}{\partial\vec{n}}=W_{i-N_D}-\sum_{j=1}^{i-1}\dfrac{\partial u_j}{\partial\vec{n}}-\sum_{j=i+1}^N\dfrac{\partial u_j}{\partial\vec{n}}\text{ on }\partial O_i,
\end{array}\right.
\]
if $i$ belongs to $\{N_D+1,\dots,N\}$. Summing the restrictions to $\Omega\setminus\cup_{j=1}^N\overline{O_j}$ of these limits, one concludes using linearity.
\end{proof}

\begin{remark}
A similar analysis could be carried out when Robin boundary conditions are imposed on the objects, that is
\[
\mathcal{B}_i(u)=\dfrac{\partial{u}}{\partial\vec{n}}+\beta_i\,u,\ i=1,\dots,N,
\]
where the coefficients $\beta_i$, with $i$ in $\{1,\dots,N\}$, are real numbers, giving rise to sufficient convergence condition. Indeed, one can see that Neumann boundary conditions correspond to Robin conditions with $\beta_i=0$, while Dirichlet conditions amount to the limiting case of $\beta_i$ tending to $+\infty$. However, the question of the unconditional convergence of the sequential form (or of the averaged parallel form) of the method in the case of Robin boundary conditions using the same coefficient $\beta$ for each object is, as far as we know, open.
\end{remark}

\section{Numerical experiments}\label{Numerics}
In this short section, we aim at confirming numerically the theoretical results obtained in the paper and investigating cases not handled by the previous analysis. More involved numerical tests of the method can be found in \cite{Ciaramella:2019}.

The problems solved numerically by the method of reflections are for the Laplace equation in some interior and exterior domains of $\mathbb{R}^2$, with both Dirichlet and Neumann boundary conditions imposed. A publicly available MATLAB package\footnote{Integral Equation Solver (\url{http://www.mathworks.com/matlabcentral/fileexchange/34241}) by Alexandre Munnier and Bruno Pin\c{c}on, MATLAB Central File Exchange. Retrieved February 15, 2016.} was used for the numerical computations, which relied on an integral formulation of the problem solved by the Nystr\"om method.

Note that, for problems set in bounded domains, similar results (which are not reprinted here) were obtained with a finite element code% \textsc{Freefem++} \cite{Hecht:2012} 
. However, employing this type of discretisation method with the method of reflections is not a sensible choice, as the computational effort required to solve any of the one-object problem is always higher than that of the \textcolor{black}{many-object one}, due to the fact that the mesh used for the problem is a subset of the meshes used for the subproblems (see Figure \ref{meshes} for an illustration). This observation emphasizes the fact that, as a boundary decomposition method, the method of reflections is, in practice, better suited to discretisation methods based on boundary integral representation of the solution.

\begin{figure}[h]
\centering
\includegraphics[width=0.24\textwidth]{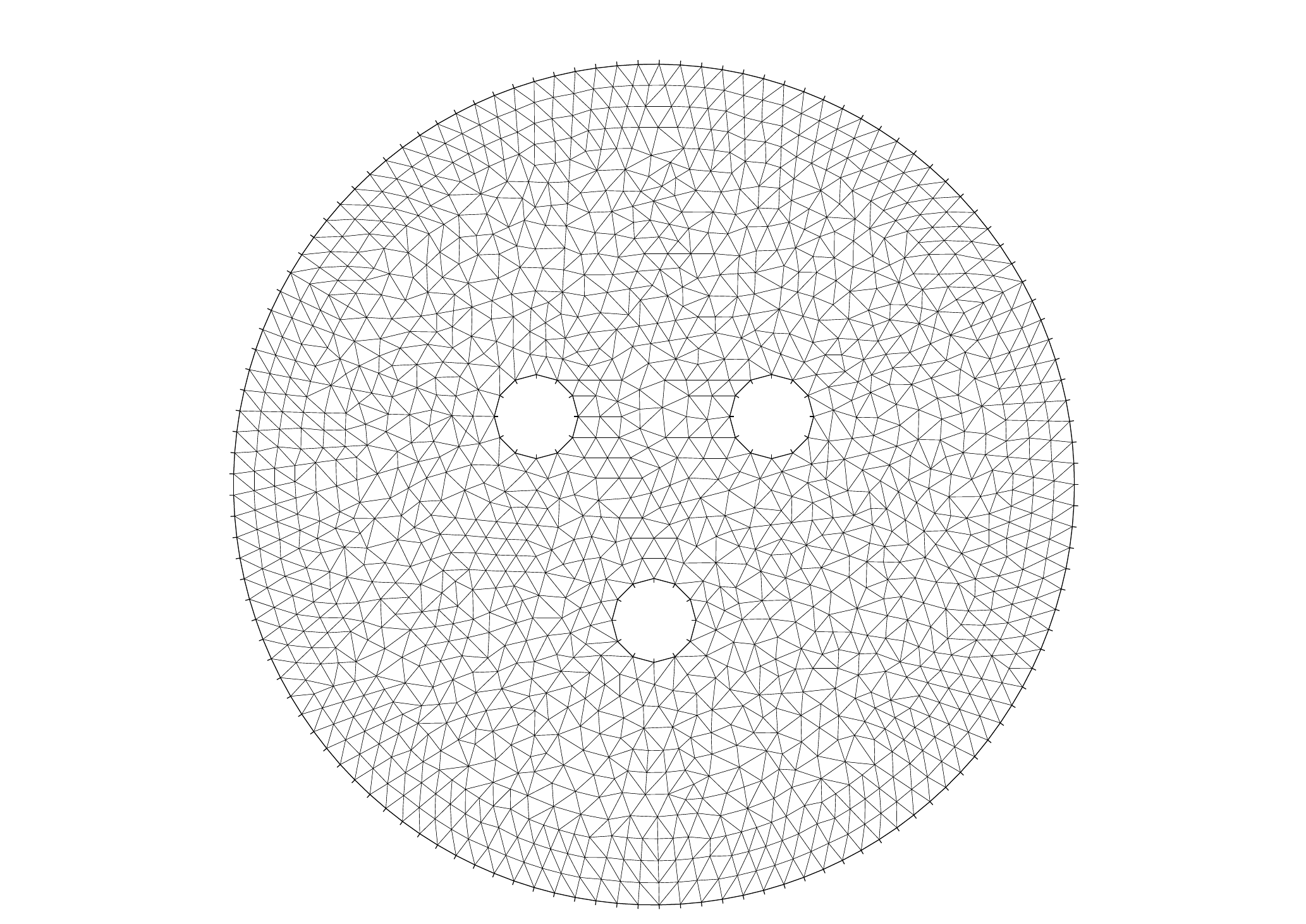}
\includegraphics[width=0.24\textwidth]{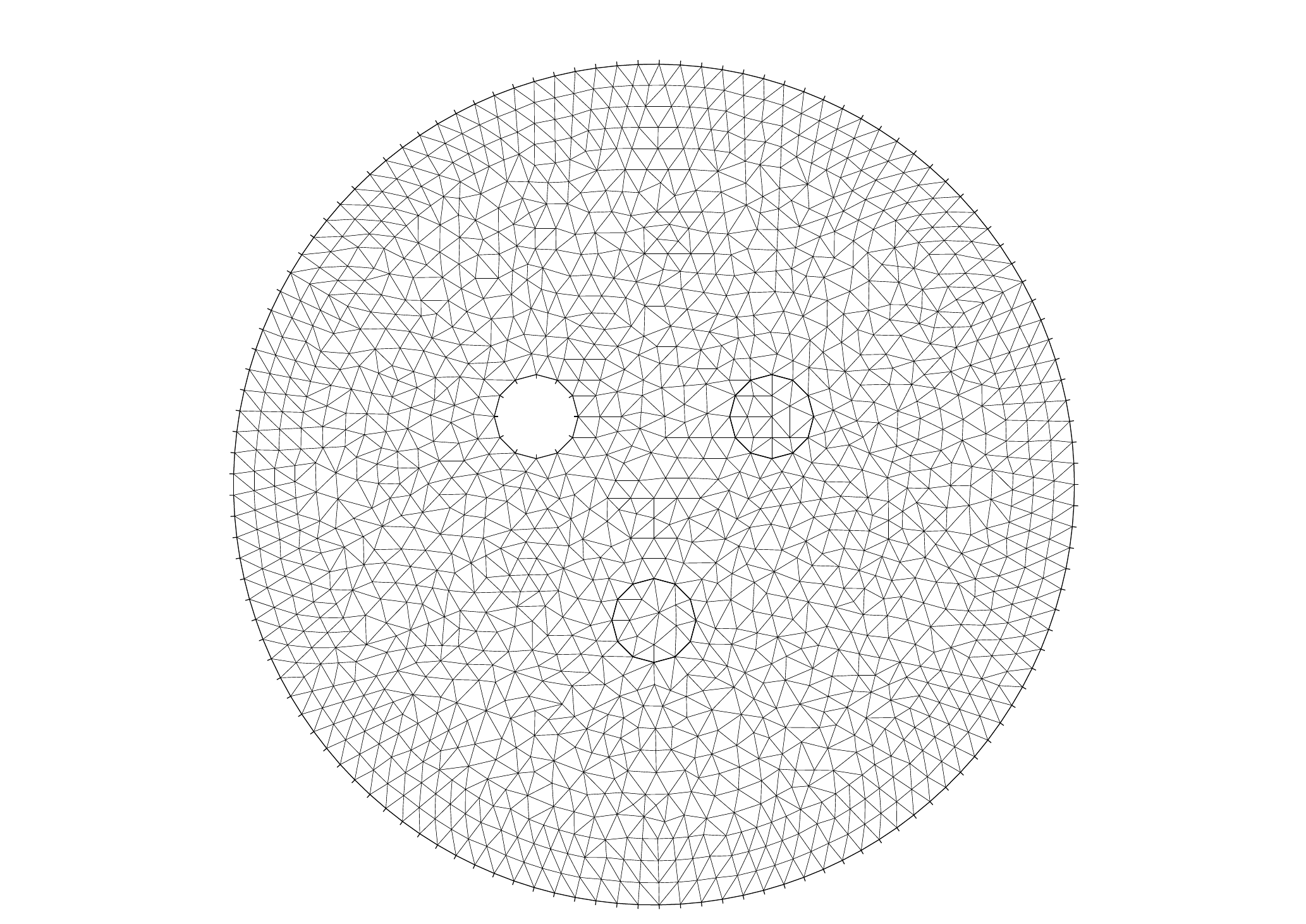}
\includegraphics[width=0.24\textwidth]{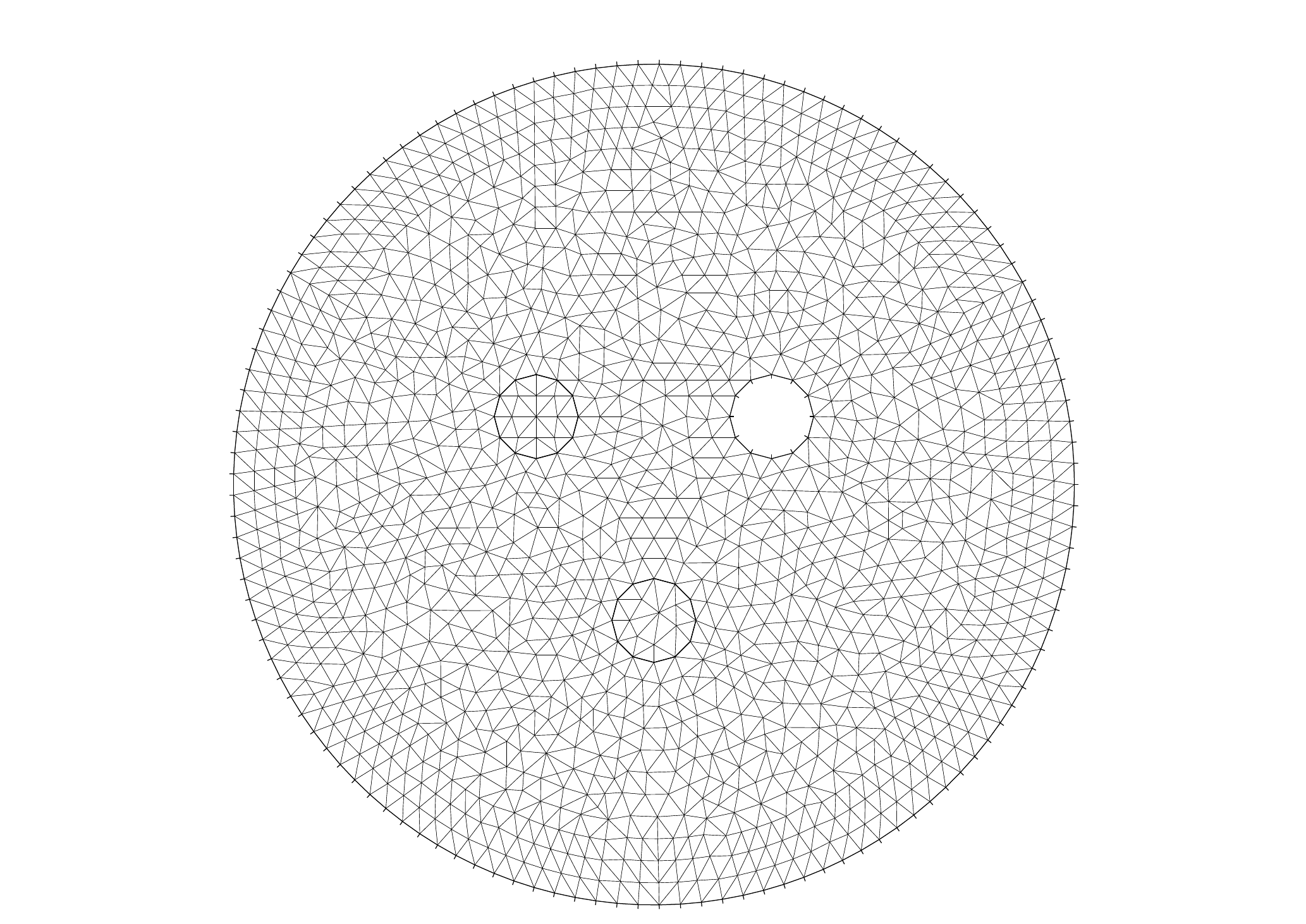}
\includegraphics[width=0.24\textwidth]{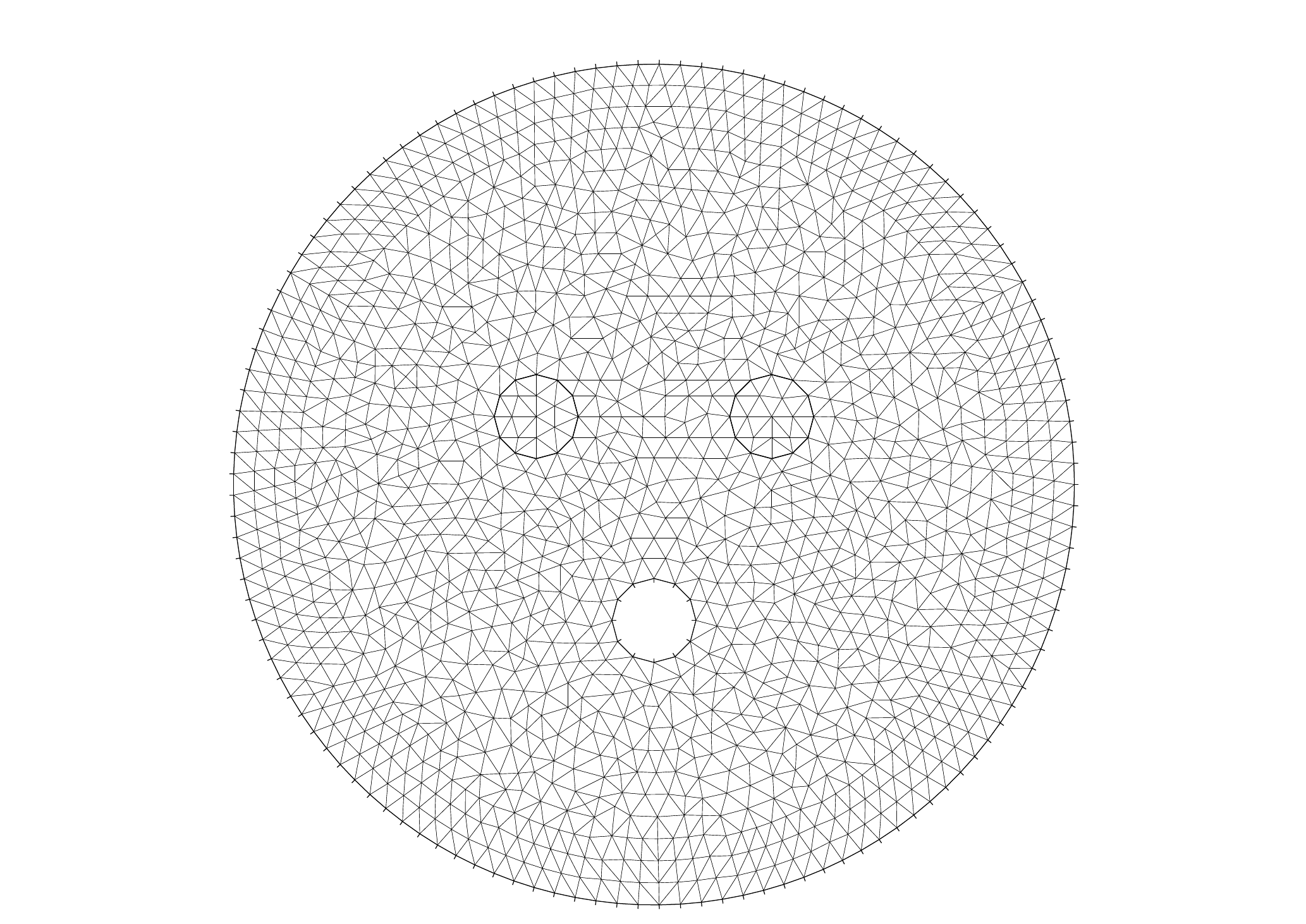}
\caption{From left to right, finite element meshes for both the problem dealt with in Subsection \ref{subsection: rate of convergence} and its three associated one-object subproblems.}\label{meshes}
\end{figure}

\subsection{Rate of convergence in a bounded domain}\label{subsection: rate of convergence}
The first numerical experiment concerns the rate of convergence of the method. It is inspired by a counterexample to the convergence of the parallel form of the method of reflections found in \cite{Ichiki:2001}. In two dimensions, we consider a bounded domain, namely a ball centred at the origin and with radius equal to $10$, containing three objects, which are balls with identical radii equal to $1$ and centres respectively set at the vertices of an equilateral triangle whose centroid lies at the origin. Using the length $l$ of a side of the triangle as a parameter, we investigate the convergence of the different forms of the method for solving a Dirichlet problem for the Laplace equation as $l$ varies.

Figure~\ref{fig:var_r} presents the relative error of the method as a function of the number of cycles for three distinct values of the parameter $l$. This relative error is based on the $\ell^2$-norm of the difference between the numerical solution for the full problem with that of the method of reflections after a given number of cycles computed at a finite number of points in the domain $\Omega\setminus\cup_{i=1}^3\overline{O_i}$. One can observe the sequential and the averaged parallel forms of the method are convergent in each case, as predicted by the theoretical results. The parallel form is seen to diverge for the smallest chosen value of $l$, but converges for larger values of the parameter.

\begin{figure}[h]%Obtained with BEM_JS6.m and FiguresBEM4.m
\centering
\subfigure[$l=1.2$]{
\begin{tikzpicture}
  \begin{axis}[width=0.33\textwidth,font=\tiny,
    tick label style={scale=0.8},
    xmin=0,xmax=100,
    xtick={0,20,40,60,80,100},
    x label style={at={(axis description cs:0.5,0.1)}},
    xlabel={Number of cycles},
    ymode=log,
    ymin=1e-5,ymax=1e25,
    ytick={1e-5,1,1e5,1e10,1e15,1e20,1e25},
    y label style={at={(axis description cs:0.08,0.5)}},
    ylabel={Error},
    legend style={at={(0.03,0.97)},anchor=north west},
    legend cell align=left]
    \addplot [mark repeat=5,mark=*,color=blue] table[x={iter},y={error_seq}]{errorl12.txt};
    \addlegendentry[scale=0.7]{sequential}
    \addplot [mark repeat=5,mark=square*,color=green] table[x={iter},y={error_par}]{errorl12.txt};
    \addlegendentry[scale=0.7]{parallel}
    \addplot [mark repeat=5,mark=triangle*,color=red] table[x={iter},y={error_avgpar}]{errorl12.txt};
    \addlegendentry[scale=0.7]{av. parallel}
  \end{axis}
\end{tikzpicture}}
\subfigure[$l=4$]{
\begin{tikzpicture}
  \begin{axis}[width=0.33\textwidth,font=\tiny,
    tick label style={scale=0.8},
    xmin=0,xmax=100,
    x label style={at={(axis description cs:0.5,0.1)}},
    xlabel={Number of cycles},
    xtick={0,20,40,60,80,100},
    ymode=log,
    ymin=1e-15,ymax=1e5,
    ytick={1e-15,1e-10,1e-5,1,1e5},
    y label style={at={(axis description cs:0.08,0.5)}},
    ylabel={Error},
    legend style={at={(0.97,0.97)},anchor=north east},
    legend cell align=left]
    \addplot [mark repeat=5,mark=*,color=blue] table[x={iter},y={error_seq}]{errorl40.txt};
    \addlegendentry[scale=0.7]{sequential}
    \addplot [mark repeat=5,mark=square*,color=green] table[x={iter},y={error_par}]{errorl40.txt};
    \addlegendentry[scale=0.7]{parallel}
    \addplot [mark repeat=5,mark=triangle*,color=red] table[x={iter},y={error_avgpar}]{errorl40.txt};
    \addlegendentry[scale=0.7]{av. parallel}
  \end{axis}
\end{tikzpicture}
}
\subfigure[$l=8$]{
\begin{tikzpicture}
  \begin{axis}[width=0.33\textwidth,font=\tiny,
    tick label style={scale=0.8},
    xmin=0,xmax=100,
    xtick={0,20,40,60,80,100},
    x label style={at={(axis description cs:0.5,0.1)}},
    xlabel={Number of cycles},
    ymode=log,
    ymin=1e-15,ymax=1e5,
    ytick={1e-15,1e-10,1e-5,1,1e5},
    y label style={at={(axis description cs:0.08,0.5)}},
    ylabel={Error},
    legend style={at={(0.97,0.97)},anchor=north east},
    legend cell align=left]
    \addplot [mark repeat=5,mark=*,color=blue] table[x={iter},y={error_seq}]{errorl80.txt};
    \addlegendentry[scale=0.7]{sequential}
    \addplot [mark repeat=5,mark=square*,color=green] table[x={iter},y={error_par}]{errorl80.txt};
    \addlegendentry[scale=0.7]{parallel}
    \addplot [mark repeat=5,mark=triangle*,color=red] table[x={iter},y={error_avgpar}]{errorl80.txt};
    \addlegendentry[scale=0.7]{av. parallel}
  \end{axis}
\end{tikzpicture}
}

%\subfigure[$l=1.2$]{\resizebox{0.32\textwidth}{!}{\includegraphics[width=.32\textwidth]{FiguresBEM4_Error_R_1_19.pdf}}}
%\subfigure[$l=4$]{\resizebox{0.32\textwidth}{!}{\includegraphics[width=.32\textwidth]{FiguresBEM4_Error_R_4_0.pdf}}}
%\subfigure[$l=8$]{\resizebox{0.32\textwidth}{!}{\includegraphics[width=.32\textwidth]{FiguresBEM4_Error_R_8_0.pdf}}}
\caption{Relative error of the method of reflections $\ell^2$-norm as a function of the number of cycles for three chosen values of the distance between the objects.}\label{fig:var_r}
\end{figure}
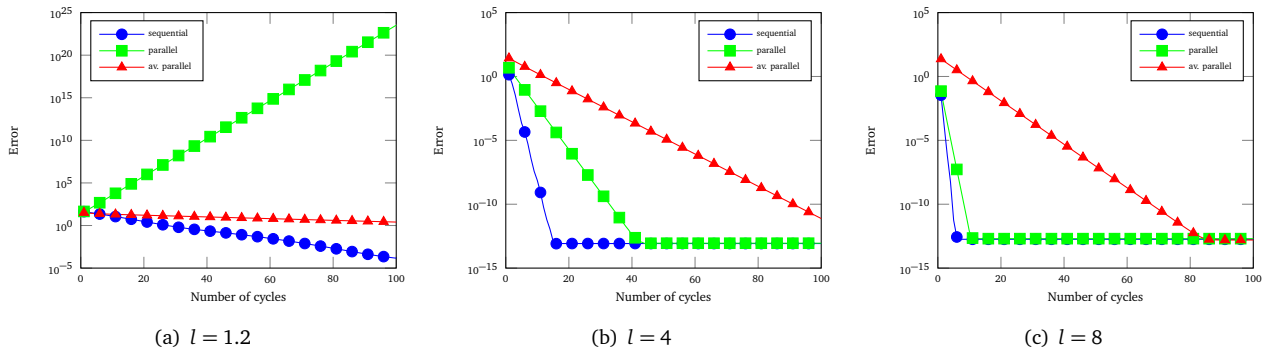

We note that the convergence is linear and that its rate increases with the value of $l$. The rate of convergence of the sequential form is also always higher than that of the averaged parallel form. This is not a surprise, as the convergence of the method of alternating projections is generally\footnote{This fact can be commonly observed in practice (see \cite{Censor:2012} for instance) and is theoretically proved for $N=2$ in \cite{Reich:2017}.} faster than that of its simultaneous counterpart. Also, the equal weights used in the averaging of the reflections may not constitute an optimal choice with respect to the rate of convergence. Moreover, from an effective computational perspective, the sequential form of the method may not be the most efficient one when a large number of objects is involved, since the implementation of the parallel variants can be achieved in practice using parallelisation. Nevertheless, a quantitative study of the trade-off between the parallelisation and the rates of convergence of the different forms of the method is out of the scope of the present work.

\subsection{Mixed boundary conditions: a case of divergence of the sequential form}\label{divergence example}
Since Theorem~\ref{convergence sequential} only ensures unconditional convergence of the sequential form of the method of reflections in an orthogonal setting, cases of divergence are expected for boundary value problems falling outside of this framework, but may prove elusive. Such a configuration was obtained in two dimensions by considering a ball of radius equal to $2$ centred at the origin, on which a Dirichlet boundary condition with constant datum is imposed, %equal to $b_1=2$
surrounded by a C-shaped set, on which a Neumann boundary condition with constant datum is imposed, % equal to $b_2=-1$
both contained in the bounded domain previously considered. The setting of this example is shown in Figure~\ref{div}.% The boundaries of the objects are uniformly discretised, with $81$ and $307$ grid points respectively. For the latter object, we consider constant Neumann boundary condition equal to $-1$.

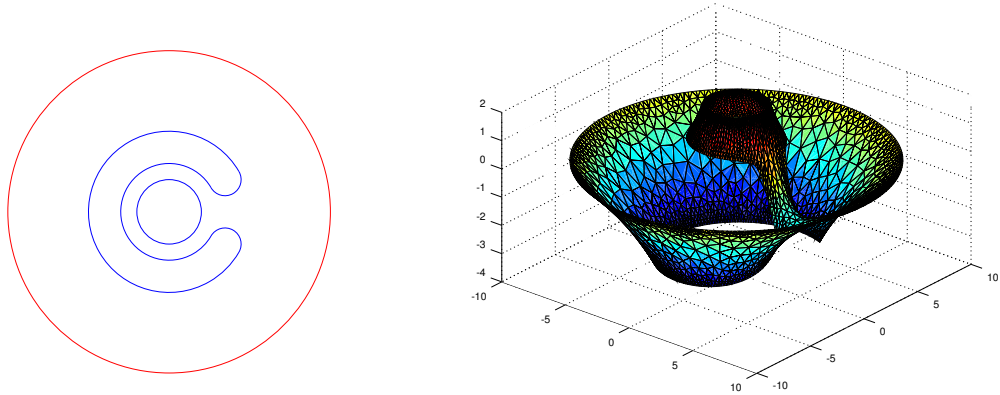
\begin{figure}[h]%Obtained with BEM_JS10.m
\centering
\resizebox{0.25\textwidth}{!}{
\begin{tikzpicture}[scale=0.5]
\def\rOmega{10};
\draw[color=red] (0.,0.) circle (\rOmega) ;

\def\xa{0.};
\def\ya{0.};
\def\raa{3.};
\def\rab{5.};
\def\dtheta{30};
\draw[color=blue] ({\xa+\raa*cos(\dtheta)},{\ya+\raa*sin(\dtheta)}) arc (\dtheta:360-\dtheta:\raa) ;
\draw[color=blue] ({\xa+\rab*cos(\dtheta)},{\ya+\rab*sin(\dtheta)}) arc (\dtheta:360-\dtheta:\rab) ;
\draw[color=blue] ({\xa+\raa*cos(\dtheta)},{\ya+\raa*sin(\dtheta)}) arc (180+\dtheta:360+\dtheta:{0.5*(\rab-\raa)}) ;
\draw[color=blue] ({\xa+\raa*cos(360-\dtheta)},{\ya+\raa*sin(360-\dtheta)}) arc (180-\dtheta:-\dtheta:{0.5*(\rab-\raa)}) ;

\def\xb{0.};
\def\yb{0.};
\def\rb{2};
\draw[color=blue] (\xb,\yb) circle (\rb);
\draw[color=white] (0,-13) circle(0.1);
\end{tikzpicture}
}
\hspace{10mm}
\includegraphics[width=0.46\textwidth]{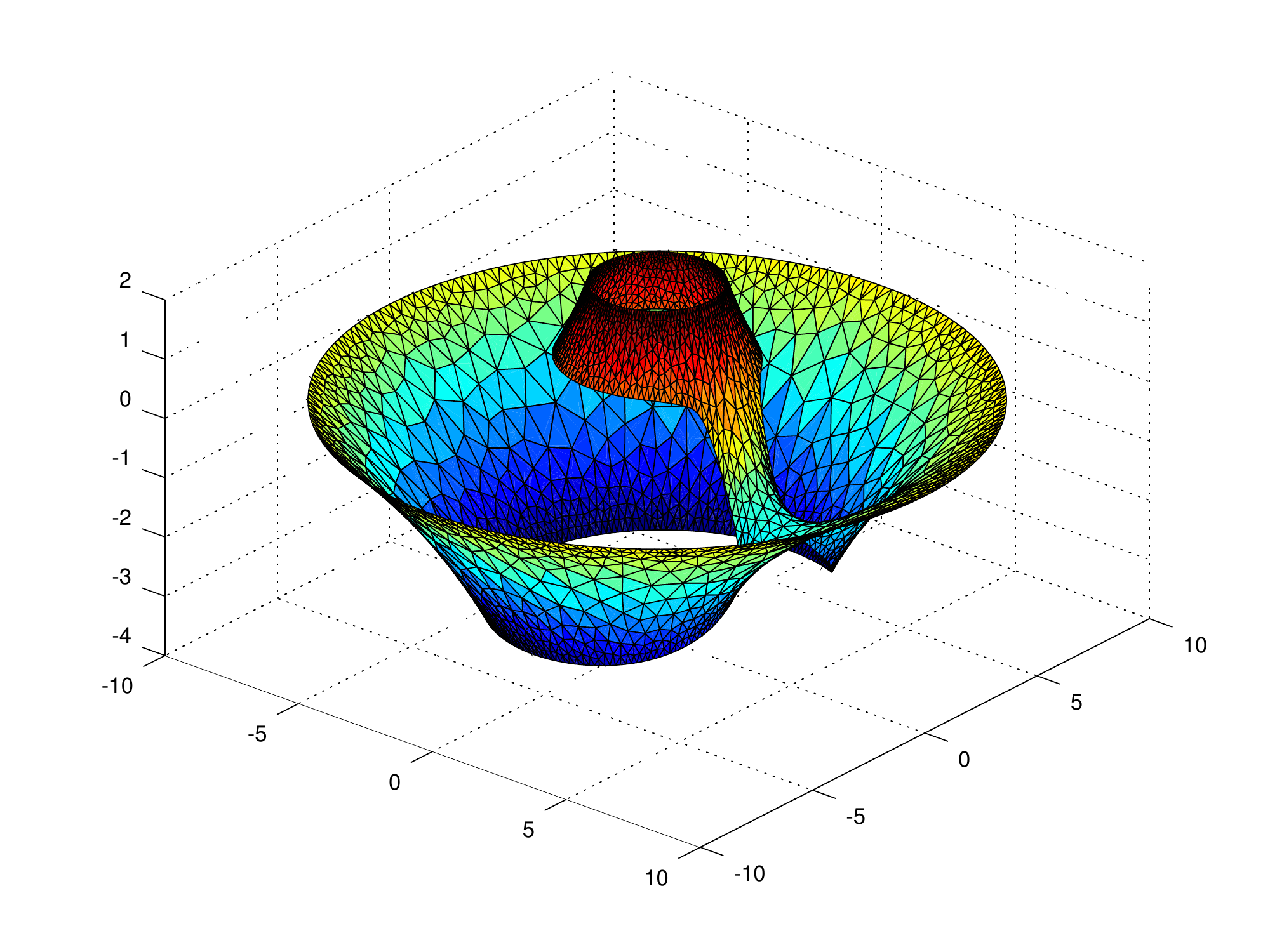}
\caption{Example of configuration giving rise to a divergence of the method of reflections when applied to the solution of a Laplace problem with mixed (Dirichlet and Neumann) boundary conditions. Left figure: representation of the domain (with red boundary) and two objects (with blue boundaries). Right figure: numerical representation of the computed reference solution for this problem.}\label{div}
\end{figure}

In the numerical experiments, none of the three forms of the method converged for such a configuration.

\subsection{Influence of the distance in an unbounded domain}
The asymptotic behaviour of the rate of convergence of the method seen as a function of the distance between the objects is finally investigated. To do this, the Laplace problem dealt with in Subsection \ref{subsection: rate of convergence} is recast as an exterior problem set in $\mathbb{R}^2$, considering Neumann boundary conditions satisfying instead of Dirichlet ones, chosen such a way that the solution tends to zero at infinity. % the form $\cos(\theta)$, with $\theta=0,\dots,2\pi$ the angle associated with the usual parametrization of the circle. The objects are discretised uniformly with $41$ points.
The distance $l$ between the objects being fixed, the corresponding contraction coefficient $K(l)$ is defined by
\[
K(l)=\lim_{k\to +\infty}\frac{\|u^{(k+1)}-u\|}{\|u^{(k)}-u\|},
\]
where $\|\cdot\|$ %corresponds to the $\ell^2$-norm of the function evaluated at a finite number of points on the boundaries of the objects.
is the $\ell^2$-norm of a set of values of the considered function at a finite number of points on the boundaries of the objects, and estimated in practice by fitting the error as a function of the iteration. The results are presented in Figure~\ref{conv}. The convergence rate of the averaged parallel form appears to be asymptotically independent of the distance between the objects, but the theoretical proof of such a result is an open question.

\begin{figure}[h]%Obtained with BEM_JS8.m and FiguresBEM3.m
\centering
\begin{tikzpicture}[]
  \begin{axis}[width=0.46\textwidth,font=\tiny,
    tick label style={scale=0.8},
    xmode=log,
    xmin=2,xmax=10,
    xtick={2,10},
    xlabel={Distance between the objects},
    ymode=log,
    ymin=1e-3,ymax=1,
    ylabel={Contraction factor},
    legend style={at={(0.03,0.03)},anchor=south west},
    legend cell align=left]
    \addplot [mark repeat=5,color=blue,mark=*] table[x={R},y={seq_coef}]{coeff_dist.txt};
    \addlegendentry[scale=0.7]{seqential slope = $-3.2$}
    \addplot [mark repeat=5,color=green,mark=square*] table[x={R},y={par_coef}]{coeff_dist.txt};
    \addlegendentry[scale=0.7]{paralell slope = $-2.09$}
    \addplot [mark repeat=5,color=red,mark=triangle*] table[x={R},y={avgpar_coef}]{coeff_dist.txt};
    \addlegendentry[scale=0.7]{av. paralell slope = $0.01$}
    
    \addplot[black,domain=2.354:10] {2.889698*exp(-3.196174*ln(x))};
    \addplot[black,domain=2.354:10] {1.676215*exp(-2.087753*ln(x))};
    \addplot[black,domain=2.354:10] {exp(0.014953*ln(x)-0.454782)};
  \end{axis}
\end{tikzpicture}
\caption{Plot of the value of the estimated contraction factor of the method as a function of the distance between the objects. The linear regressions are done using values associated with the five largest considered distances and are represented by black solid lines.}\label{conv}
\end{figure}
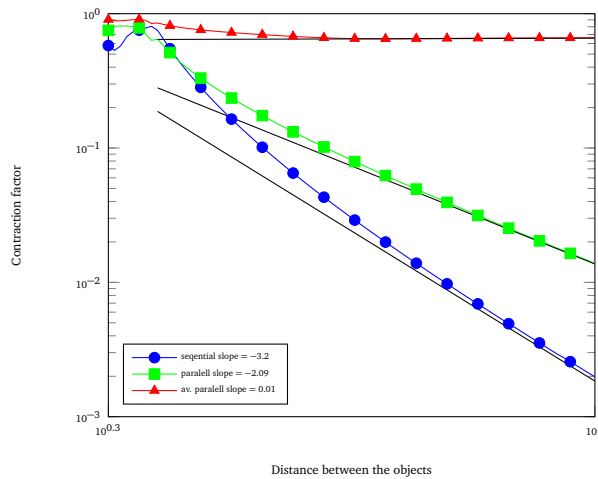

More generally, scalability issues, that is the analysis of the rate of convergence as a function of the number of the objects involved in the problem, is out of the scope of this paper. Note that settings for which a decomposition method achieves scalability have already been exhibited for Schwarz-type methods, see, e.g., \cite{Ciaramella:2017}, but remains an open question in the context of the method of reflections.

\subsection*{Acknowledgements}
Philippe Laurent would like to thank Fr\'ed\'eric Boyer for introducing him to the method of reflections. Guillaume Legendre would like to thank Christophe Hazard for pointing him to the relevant paper \cite{Balabane:2004}, and Mikhael Balabane himself for an interesting discussion on the topic. Julien Salomon would like to thank Gabriele Ciaramella, Olivier Glass and Alexandre Munnier for helpful discussions. Finally, the authors collectively thank the anonymous reviewers whose comments and suggestions helped improve the manuscript.
 
\printbibliography
\end{document}